\documentclass{gen-j-l}

\usepackage{amsmath}
\usepackage{amsfonts}
\usepackage{amssymb}
\usepackage{mathtools}
\usepackage{enumitem}
\usepackage{esint}
\usepackage{hyperref}

\newtheorem{theorem}{Theorem}[section]
\newtheorem*{theorem*}{Theorem}
\newtheorem{corollary}[theorem]{Corollary}
\newtheorem{lemma}[theorem]{Lemma}
\newtheorem{prop}[theorem]{Proposition}

\theoremstyle{definition}
\newtheorem{definition}[theorem]{Definition}
\newtheorem{example}[theorem]{Example}

\theoremstyle{remark}
\newtheorem{remark}[theorem]{Remark}

\def \R {\mathbb{R}}
\def \E {\mathbb{E}}
\def \P {\mathbb{P}}
\def \A {\mathcal{A}}
\def \U {\mathcal{U}}
\def \N {\mathbb{N}}

\def \comp {\circ}
\def \sbs {\subseteq}
\def \sps {\supseteq}
\def \cross {\times}
\def \res {\big|_}
\def \eps {\epsilon}

\def \eps {\epsilon}

\def \op {\text{op}}

\newcommand{\Lipo}[1] {\text{Lip}_0\left( #1 \right)}
\newcommand{\Lip}[1] {\text{Lip}\left( #1 \right)}

\newcommand{\dev}[1] {\text{dev}\left( #1 \right)}
\newcommand{\height}[1] {\text{ht}\left( #1 \right)}

\def \sbs {\subseteq}
\def \sps {\supseteq}

\usepackage{graphicx}

\numberwithin{equation}{section}

\begin{document}

\title{Thick Families of Geodesics and Differentiation}


\author{Chris Gartland}
\thanks{Thanks to Jeremy Tyson and Mikhail Ostrovskii for helpful comments in the preparation of this article.}





\begin{abstract}
The differentiation theory of Lipschitz functions taking values in a Banach space with the Radon-Nikod\'ym property (RNP), originally developed by Cheeger-Kleiner, has proven to be a powerful tool to prove non-biLipschitz embeddability of metric spaces into these Banach spaces. Important examples of metric spaces to which this theory applies include nonabelian Carnot groups and Laakso spaces. In search of a metric characterization of the RNP, Ostrovskii found another class of spaces that do not biLipschitz embed into RNP spaces, namely spaces containing thick families of geodesics. Our first result is that any metric space containing a thick family of geodesics also contains a subset and a probability measure on that subset which satisfies a weakened form of RNP Lipschitz differentiability. A corollary is a new nonembeddability result. Our second main result is that, if the metric space is a nonRNP Banach space, a subset consisting of a thick family of geodesics can be constructed to satisfy true RNP differentiability. An intriguing question is whether this differentiation criterion, or some weakened form of it such as the one we prove in the first result, actually characterizes general metric spaces non-biLipschitz embeddable into RNP Banach spaces.
\end{abstract}

\maketitle


\tableofcontents

\section{Introduction}
\subsection{Historical Background}
\label{ss:histback}
In \cite{Ch99}, Cheeger introduced the notion of a Lipschitz differentiable structure on a metric measure space and proved that every doubling space satisfying a Poincar\'e inequality (henceforth PI space) admits one. Such a structure allows one to differentiate real-valued Lipschitz functions almost everywhere with respect to an atlas of $\R^k$-valued Lipschitz functions ($k$ can vary from chart to chart). We'll call metric measure spaces admitting this structure \emph{Lipschitz differentiability spaces}. Cheeger notes in Theorem 14.3 of \cite{Ch99} that Lipschitz differentiability spaces which biLipschitz embed into a finite dimensional Euclidean space must have every tangent cone (see Section \ref{sec:prelims} for background on tangent cones) at almost every point be biLipschitz equivalent to $\R^k$ (the same $\R^k$ that the chart containing the point maps to) and that nonabelian Carnot groups and Laakso spaces violate this condition. In \cite{CK09}, Cheeger and Kleiner generalize the result from \cite{Ch99} and prove that PI spaces admit RNP Lipschitz differentiable structures, which generalize Lipschitz differentiable structures in the sense that every Lipschitz function taking values in any Banach space with the RNP is differentiable almost everywhere. (A Banach space $B$ has the RNP if any Lipschitz map $\R \to B$ is differentiable Lebesgue-almost everywhere, or, equivalently, if for every probability space $\Omega$ and for every martingale $M_n: \Omega \to B$, if $\sup_n \|M_n\|_{L^\infty(\Omega;B)} < \infty$, then $M_n$ converges almost surely. For reference on RNP spaces, see Chapter 2 of \cite{Pi16}, specifically Theorem 2.9 and Remark 2.17.) Such metric measure spaces will be referred to as \emph{RNP Lipschitz differentiability spaces}. In Theorem 1.6 of \cite{CK09}, the authors note that, again, any such metric measure space biLipschitz embedding into an RNP space must have every tangent cone at almost every point be biLipschitz equivalent to $\R^k$. Thus, nonabelian Carnot groups and Laakso spaces do not biLipschitz embed even into any RNP space. For a metric measure space, we call the phenomenon of admitting an RNP differentiable structure and violating the condition the every tangent cone at almost every point is biLipschitz equivalent to $\R^k$ the \emph{differentiation nonembeddability criterion} into RNP spaces (occasionally, we will also use the terms \emph{true} RNP differentiable structure or \emph{true} differentiation nonembeddability criterion to distinguish them from a weaker version we introduce in Theorem \ref{thm:mainthmssummary}).

It's been known since at least 1973 that Lipschitz maps from separable Banach spaces to RNP spaces are, in a suitable sense, differentiable almost everywhere. This is due independently to Aronszajn \cite{Ar76}, Christensen \cite{Chr73}, and Mankiewicz \cite{Ma73} (see section 6.6 of \cite{BL00}). It follows that the RNP is inherited under biLipschitz embeddability of Banach spaces, since it is inherited under linear-biLipschitz embeddability. It is then natural to ask for a purely metric characterization of the RNP - one that does not rely on the linear structure. This question was asked by Bill Johnson in 2009 and answered in 2014 by Ostrovksii (see \cite{Os14b}) with the following theorem:

\begin{theorem}[Corollary 1.5 \cite{Os14b}] \label{thm:ostrovskii1}
A Banach space does not have the RNP if and only if it contains a biLipschitz copy of a metric space containing a thick family of geodesics.
\end{theorem}

\noindent In particular, a new criterion for non-biLipschitz embeddability of metric spaces into RNP Banach spaces was discovered. Ostrovskii went on to give a simple proof that Laakso spaces contain thick families of geodesics. This turned out to be a much shorter and more natural way to prove their nonembeddability into RNP spaces compared to the differentiation nonembeddability criterion.

On the other hand, according to another intriguing result of Ostrovskii, no Carnot group can contain a thick family of geodesics. This is due to Li's proof (section 7.1 of \cite{Li14}) of the nontrivial Markov convexity of Carnot groups, the fact that Markov convexity is inherited under biLipschitz embeddings, and the following result of Ostrovskii (see Definition \ref{def:thickfamily}) :

\begin{theorem}[Theorem 1.5 of \cite{Os14c}] \label{thm:ostrovskii2}
Metric spaces containing a thick family of geodesics have no nontrivial Markov convexity.
\end{theorem}

\noindent So although containing a thick family of geodesics is a necessary condition for the non-biLipschitz embeddability of Banach spaces into RNP Banach spaces, the same is not true of general metric spaces, even geodesic metric spaces such as Carnot groups. Our motivation for this article is to study the relationship between these two criteria for non-biLipschitz embeddability into RNP spaces, and, more generally, to find a characterization of metric spaces that do not embed into RNP spaces. To this end, we prove two results: the first is that any metric space containing a thick family of geodesics also contains a subset and a probability measure on that subset which satisfies a weakened form of the differentiation nonembeddability criterion (see Theorem \ref{thm:mainthmssummary}). Our second result, Theorem \ref{thm:gendiamembed}, is that any nonRNP Banach space contains a biLipschitz copy of a metric measure space satisfying the true differentiation nonembeddability criterion.

\subsection{Summary of Results and Discussion of Proof Methods}
\subsubsection{Summary of Results}
The type of differentiable structure we construct is weaker than the true RNP differentiable structure because the almost everywhere approximation of RNP-valued Lipschitz functions by their derivative only holds on some sequence of scales tending to 0 instead of all scales. More specifically, we prove Theorems \ref{thm:Sinfty} and \ref{thm:weakRNPdiff}, which can be summarized as:

\begin{theorem}[Summary of Theorems \ref{thm:Sinfty} and \ref{thm:weakRNPdiff}] \label{thm:mainthmssummary}
For any complete metric space $M$ containing a thick family of geodesics, there exist a compact subset $X_\infty$, Borel probability measure $\mu_\infty$ on $X_\infty$, Lipschitz map $\pi: X_\infty \to [0,1]$, Borel subset $S_\infty \sbs X_\infty$, a sequence of scales $r_i(x) \searrow 0$ for almost every $x \in X_\infty$, and a nonprincipal ultrafilter $\U(x)$ on $\N$ for each $x \in S_\infty$ such that:
\begin{itemize}
\item[\ref{thm:Sinfty}] $\mu_\infty(S_\infty) > 0$, and for every $x \in S_\infty$ the tangent cone $T_x^{r_i(x),\U(x)}X_\infty$ does not topologically embed into $\R$.
\item[\ref{thm:weakRNPdiff}] For every RNP space $B$ and Lipschitz map $f: X_\infty \to B$, for $\mu_\infty$-almost every $x \in X_\infty$, $f$ is differentiable at $x$ with respect to $\pi$ along the sequence of scales $(r_i(x))_{i=0}^\infty$.
\end{itemize}
\end{theorem}

\noindent The map $\pi$ is the single chart in the weak RNP Lipschitz differentiable atlas. As a corollary, we obtain a new proof of nonembeddability into RNP spaces:

\begin{corollary} \label{cor:RNPapplication}
A metric space $M$ containing a thick family of geodesics does not biLipschitz embed into any RNP space.
\end{corollary}

\noindent The proof is the same as for the differentiation nonembeddability criterion into RNP spaces, Theorem 1.6 from \cite{CK09}.

\begin{proof}
Let $B$ be an RNP space and assume there is a biLipschitz map $g: M \to B$. We may assume $M$ is complete. Let $X_\infty \sbs M$, $\mu_\infty$, $S_\infty$, $r_i(x)$, and $\U(x)$ be as in the statement of the theorem. Since $\mu_\infty(S_\infty) > 0$, there exist a point $x \in S_\infty$ and a nonprincipal ultrafilter $\U(x)$ such that $f$ is differentiable at $x$ along $(r_i(x))_{i=0}^\infty$ with respect to $\pi$ and $T_x^{r_i(x),\U(x)}X_\infty$ does not topologically embed into $\R$. $f$ being differentiable with respect to $\pi$ at $x$ along $(r_i(x))_{i=0}^\infty$ implies that there exists a unique linear map $f'(x): \R \to B$ such that, for every nonprincipal ultrafilter $\U$, the blowup of $f$ at $x$, $f_x: T_x^{r_i(x),\U}X_\infty \to B$, exists and factors though the blowup of $\pi$ at $x$, $\pi_x: T_x^{r_i(x),\U}X_\infty \to \R$, and $f'(x): \R \to B$. That is, $f_x = f'(x) \comp \pi_x$. Since $T_x^{r_i(x),\U(x)}X_\infty$ does not topologically embed into $\R$, $\pi_x$ cannot be biLipschitz, which  by the factorization implies $f_x$ cannot be biLipschitz, in turn implying $f$ cannot be biLipschitz.
\end{proof}

Theorem \ref{thm:mainthmssummary} actually proves a stronger statement, Corollary \ref{cor:Carnotapplication}. We postpone the proof till section \ref{sec:application}. We chose to give a separate proof Corollary \ref{cor:RNPapplication} because it is easier and requires no knowledge of Carnot groups.

\begin{corollary} \label{cor:Carnotapplication}
A complete metric space $M$ containing a thick family of geodesics does not biLipschitz embed into the product metric space $G \cross B$, where $G$ is a Carnot group and $B$ is an RNP space.
\end{corollary}

At the time of this writing, Theorems \ref{thm:ostrovskii1} and \ref{thm:ostrovskii2} were the only known nontrivial means by which one could prove nonembeddability of thick families of geodesics into metric spaces. Suppose $G$ is a nonabelian Carnot group, such as the Heisenberg group, and $B$ is an RNP which is not super-reflexive, such as $\ell^1$. Then $G$ embeds into no RNP space by the differentiation nonembeddability criterion, so Theorem \ref{thm:ostrovskii1} does not apply to $G \cross B$, and $B$ has no nontrivial Markov convexity, so Theorem \ref{thm:ostrovskii2} does not apply to $G \cross B$. That non-super-reflexive spaces have no nontrivial Markov convexity follows from the fundamental theorem of Mendel-Naor on Markov convexity (Theorem 1.3 of \cite{MN13}), and Pisier's renorming theorem (Theorem 11.37 of \cite{Pi16}). Thus, Corollary \ref{cor:Carnotapplication} is a genuinely new nonembeddability result. 

In our second result, Theorem \ref{thm:gendiamembed}, we restrict our attention from a general metric containing a thick family of geodesics to a nonRNP Banach space $B$. This is indeed a ``restriction" since every such $B$ contains a thick family of geodesics by \cite{Os14b}, as previously stated. In this setting, we prove that the subset $X_\infty$ and measure $\mu_\infty$ can be constructed to satisfy the true RNP differentiation nonembeddability criterion (not just the weakened form described in Theorem \ref{thm:mainthmssummary}). That it satisfies the true RNP differentiation criterion is a consequence of the fact that it is an inverse limit of an admissible system of graphs, defined in \cite{CK15}. In that article, Cheeger and Kleiner proved that such spaces are PI spaces. They also gave a necessary and sufficient condition for these spaces to satisfy the differentiation nonembeddability criterion into RNP spaces, stated in Theorem 10.2 of \cite{CK15}. We verify this condition for our subset $X_\infty \sbs B$, and thus our result can be viewed as a converse to Theorem 10.2.

\begin{theorem} \label{thm:gendiamembed}
Every nonRNP Banach space contains a biLipschitz copy of a metric measure space satisfying the differentiation nonembeddability criterion. The metric measure space is an inverse limit of admissible graphs, as in \cite{CK15}, with nonEuclidean tangent cones at almost every point.
\end{theorem}

\subsubsection{Proof Methods} \label{ss:proofmethods}
The subset $X_\infty$ of a metric space $M$ containing a thick family of geodesics from Theorem \ref{thm:mainthmssummary} is constructed as an inverse limit of graphs. Cheeger and Kleiner proved in \cite{CK15} that inverse limits of certain ``admissible" inverse systems of graphs, such as Laakso spaces, are PI spaces and hence RNP Lipschitz differentiability spaces. It is this result which lead us to believe that $X_\infty$ could be constructed to satisfy some kind of RNP Lipschitz differentiability. However, our space $X_\infty$ cannot be constructed to be a PI space in any obvious way, and thus the theory of \cite{CK09} does not apply; we are required to construct derivatives of RNP-valued Lipschitz functions and prove their defining approximation property by hand. To do so, we use only the almost sure differentiability of Lipschitz maps $\R \to B$ and the almost sure convergence of $B$-valued martingales for RNP spaces, which are quite classical compared to the asymptotic norming property of RNP spaces used in \cite{CK09}. We also make heavy use of the uniformly topology on Banach spaces of Lipschitz functions, in contrast to the Sobolev space techniques employed in \cite{CK09} and \cite{CK15}.

Apart from these differences in proof techniques, the inverse systems of graphs we consider are fundamentally different from the admissible systems in \cite{CK15} for two reasons. Firstly, in \cite{CK15}, the graphs are equipped with geodesic metrics, and the metrics on our graphs are only geodesic along directed edge paths. In fact, the inverse limit space need not even be quasiconvex, while PI spaces are always quasiconvex. Secondly, in \cite{CK15}, the lengths of edges in the sequence of graphs decrease by a constant factor $m \geq 2$ in each stage of the sequence, independent of the stage or edge. In our graphs, the edge lengths decrease by factors going to $\infty$. We make frequent use of this rapid decay in a number of independent results, such as \eqref{eq:piLipbnd}, \eqref{eq:fiberdiam}, and Lemma \ref{lem:pi-1Lipbnd}. Loosely, the rapid decay in edge length allows us to well-control the local geometry near a point along scales proportional to the lengths of edges containing the projections of the point, at the cost of control over the geometry along other scales, which would be necessary to prove true RNP differentiability.

The uniform topology on Lipschitz algebras has been studied before within the context of Lipschitz differentiability spaces. For, example, in \cite{Sch14}, Schioppa showed how to associate a Weaver derivation (which involves continuity with respect to uniform topology) to an Alberti representation, and Alberti representations were demonstrated by Bate in \cite{Bat15} to be intimately connected to Lipschitz differentiability. Schioppa constructs the partial derivative of a function by taking its derivative along curve fragments and averaging them together with respect to the Alberti representation. Our procedure for constructing the derivative of a function (see Theorem \ref{thm:derivdef}), is very similar in nature; indeed, Lemma \ref{lem:measdisint} gives Alberti representations of $\mu_i$, which (after taking a suitable limit) give rise to an Alberti representation of $\mu_\infty$. We also note that in \cite{Bat15}, Bate gives necessary and sufficient conditions for a collection of Alberti representations to induce a Lipschitz differentiable structure on a metric measure space using what he called \emph{universality} (see Definition 7.1 from \cite{Bat15}). Our representation from Lemma \ref{lem:measdisint} will generally fail this property (or at least doesn't obviously satisfy it - we don't actually provide an example), which is consistent with our discussion that the space $(X_\infty,d,\mu_\infty)$ is not a true Lipschitz differentiability space (again, we don't actually provide an example of this). We believe it is possible to find a weakened form of universality corresponding to the weakened form of differentiation from Theorem \ref{thm:weakRNPdiff}.

The construction of the inverse limit of admissible graphs, $X_\infty$, of Theorem \ref{thm:gendiamembed} is achieved by fine tuning two of the aspects of Ostrovskii's construction of a thick family of geodesics in nonRNP spaces. His construction is also essentially an inverse limit of a system of graphs, but the system is not ``admissible" in the sense of \cite{CK15} for two reasons. Firstly, the metrics on his system are not uniformly quasiconvex, which is a necessary condition for a metric space to be a PI space. Secondly, the lengths of edges in a graph in an admissible system must be constant, but in the system of \cite{Os14b}, the ratio of lengths of two edges in a graph may become unbounded.

The second obstacle is easily overcome in the following way: the length of an edge in a graph in the system from \cite{Os14b} corresponds to the coefficient $\alpha_i$ of some convex combination $z = \alpha_1z_1 + \dots \alpha_n z_n$ with $\|z - z_i\| > \delta$ and $\|z\|, \|z_i\| < 1$. By density of the dyadic rationals in $(0,1)$, we may make small adjustments $z_i \to z_i'$ to obtain $z = q_1 z_1' + \dots q_n z_n'$ with each $q_i$ a dyadic rational, all while maintaining $\|z - z_i'\| > \delta$ and $\|z\|, \|z_i'\| < 1$. We then `split up' the convex combination into terms whose coefficients have numerator equal to 1. For example, $\frac{1}{2}z_1' + \frac{1}{4}z_2' + \frac{1}{4}z_3' \to \frac{1}{4}z_1' + \frac{1}{4}z_1' + \frac{1}{4}z_2' + \frac{1}{4}z_3'$. The edges corresponding to this convex combination now all have length $\frac{1}{4}$. The first obstacle can be overcome by constructed $X_i$ with rapidly decreasing edge length, similar to construction in the proof of Theorem \ref{thm:existinvsys}. Using the rapid decrease in edge length to control the quasiconvexity of the graphs is similar to the proof of Lemma \ref{lem:piLipbnd}.

\subsection{Outline}
Section \ref{sec:prelims} sets notation and terminology and defines thick families of geodesics.

Sections \ref{sec:axioms}-\ref{sec:weakRNPdiff} are concerned with the proof of Theorem \ref{thm:mainthmssummary}, Section \ref{sec:application} contains the proof of Corollary \ref{cor:Carnotapplication}, Section \ref{sec:invlimgraphsnonRNP} contains the construction of the inverse limit of graphs in nonRNP Banach spaces from Theorem \ref{thm:gendiamembed}, and Section \ref{sec:questions} contains some relevant questions.

For an efficient reading of Sections \ref{sec:axioms}-\ref{sec:weakRNPdiff}, we advise the reader to start with Section \ref{sec:axioms}, skip ahead to Section \ref{sec:weakRNPdiff}, and then refer back to the between sections as they are needed to understand the proof of Theorem \ref{thm:weakRNPdiff}.

In Section \ref{sec:axioms}, we give the axioms for thick inverse systems of graphs whose inverse limit we are able to prove the weak form of differentiation of. Also included in this section are frequently used consequences of the axioms and a proof of one of the main theorems of the article, Theorem \ref{thm:existinvsys}. This theorem asserts the existence of the thick inverse system of graphs in any metric space containing a thick family of geodesics. In Section \ref{sec:asymprops}, we define the set $S_\infty$ and prove  $\mu_\infty(S_\infty) > 0$. We also include results on asymptotic local  geometry of the graphs. Section \ref{sec:approx} covers the use of conditional expectation in approximating functions on $X_\infty$ via functions on $X_i$. Also in this section is the definition of the derivative of RNP space-valued Lipschitz functions on $X_\infty$. A relevant maximal operator and corresponding maximal inequality are defined and proved in Section \ref{sec:maximalineq}. Section \ref{sec:weakRNPdiff} contains the proof of the main theorem, Theorem \ref{thm:weakRNPdiff}, the weak form of differentiability.
  
\section{Preliminaries}
\label{sec:prelims}
Let $(M,d)$ be a metric space and $p,q \in M$. A \emph{$p$-$q$ geodesic} is an isometric embedding from some closed bounded interval into $M$ mapping the left endpoint of the interval to $p$ and the right endpoint to $q$. $d$ is said to be \emph{geodesic} if there exists a $p$-$q$ geodesic for every $p,q \in M$.

The next definition concerns thick families of geodesics. Informally, a family of geodesics is concatenation closed if for any $\gamma_1,\gamma_2$ in the family, the geodesic obtained by concatenating an initial segment of $\gamma_1$ and a terminal segment of $\gamma_2$ also belongs to the family. Informally, a concatenation closed family of $p$-$q$ geodesics is $\alpha$-thick if for any geodesic $\gamma$ in the family and any finite set of points $F$ in the image of $\gamma$, there is another geodesic $\tilde{\gamma}$ in the family that intersects $\gamma$ at each point of $F$ (but possibly more points), and so that the deviation of $\tilde{\gamma}$ from $\gamma$ between their points of intersection adds up to at least $\alpha$. 

\begin{definition} \label{def:thickfamily}
We follow \cite{Os14a}. A family of $p$-$q$ geodesics $\Gamma$ with common domain $[a,b]$ is said to be \emph{concatenation closed} if for every $c \in [a,b]$ and $\gamma_1,\gamma_2 \in \Gamma$ with $\gamma_1(c) = \gamma_2(c)$, the concatenated curve $\gamma$ defined by $\gamma(t) = \gamma_1(t)$ if $t \in [a,c]$, $\gamma(t) = \gamma_2(t)$ if $t \in [c,b]$, also belongs to $\Gamma$.

Given $\alpha > 0$, a concatenation closed family of $p$-$q$ geodesics $\Gamma$ sharing a common domain $[a,b]$ is said to be \emph{$\alpha$-thick} or an \emph{$\alpha$-thick family of geodesics} if for every $\gamma \in \Gamma$ and $a = t_0 < t_1 < \dots t_k = b$, there exist $a= q_0 < s_1 < q_1 < s_2 < \dots s_j < q_{j} = b$ and $\tilde{\gamma} \in \Gamma$ such that
\begin{itemize}
\item $\{t_i\} \sbs \{q_i\}$
\item $\gamma(q_i) = \tilde{\gamma}(q_i)$
\item $\sum_{i=1}^j d(\gamma(s_{i}),\tilde{\gamma}(s_i)) \geq \alpha$
\end{itemize}
\noindent A concatenation closed family of $p$-$q$ geodesics $\Gamma$ sharing a common domain $[a,b]$ is said to be \emph{thick} or a \emph{thick family of geodesics} if it is $\alpha$-thick for some $\alpha > 0$.
\end{definition}

\begin{example}[Laakso-Lang-Plaut Diamond Space]
The following space was inspired by constructions of Laakso in \cite{La00} but first appeared in \cite{LP01}. We inductively define a sequence of metric spaces which are graphs equipped with the path metric. $L_0$ is defined to be $I = [0,1]$; as a graph it has two vertices, 0 and 1, and one edge. $L_{i+1}$ is obtained from $L_i$ by replacing each edge of $L_i$ with a copy of the graph show in Figure \ref{fig:laakso}. The edges are weighted so that the diameter of each $L_i$ is 1. There are canonical 1-Lipschitz maps $L_{i+1} \to L_i$, and $L_\infty$ is defined to be the inverse limit of this system. The collection of all 0-1 geodesics in $L_\infty$ is a $(\frac{1}{2}-\eps)$-thick family of geodesics for every $\eps > 0$. When equipped with a certain measure, $L_\infty$ also becomes an RNP Lipschitz differentiability space with a single differentiable chart given by the canonical map $L_\infty \to L_0 = I \sbs \R$.
\end{example}

\begin{figure}
\includegraphics[scale=1]{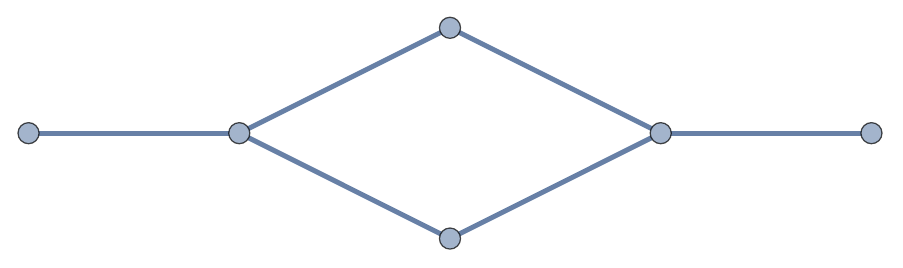}
\caption{$L_1$, the first graph of the Laakso-Lang-Plaut diamond graphs. Each new graph $L_{i+1}$ is obtained from $L_i$ by replacing each edge with a copy of $L_1$, scaled down so that the diameter of $L_{i+1}$ remains 1.}
\label{fig:laakso}
\end{figure}

Throughout, $(B,\|\cdot\|)$ will denote a general Banach space. In places where differentiation or martingale convergence is involved, it may be necessary to assume that $B$ is an RNP space, and in such cases the assumption will be stated.

Whenever $(X,\A,\mu)$ is a measure space and $p \in [1,\infty]$, the Lebesgue space $L^p(X,\A,\mu)$ of (equivalence classes of) real-valued functions will be denoted $L^p(\mu)$. When dealing with $B$-valued functions, we use the notation $L^p(\mu;B)$ (see Chapter 1 of \cite{Pi16} for background on Bochner measurable functions, integrals, and conditional expectations). 

Given two metric spaces $(X,d_X)$ and $(Y,d_Y)$ and a Lipschitz map $f: X \to Y$, we define $\Lip{f} := \sup_{x \neq y} \frac{d_Y(f(x),f(y))}{d_X(x,y)}$. For a metric space $(X,d)$ with basepoint $x_0$, define $\text{Lip}_{0}(X;B)$ to be the Banach space of Lipschitz functions $f: X \to B$ satisfying $f(x_0) = 0$, equipped with the norm $\|f\|_{\text{Lip}_{0}(X;B)}:= \text{Lip}(f)$ (when $B = \R$, we suppress notation and simply write $\Lipo{X}$). Note that, when $\text{diam}(X) \leq 1$, $\|f\|_{\text{Lip}_{0}(X;B)} \leq \|f\|_{\text{unif}}$ (we shall generally find ourselves in this situation).

Given a point $x \in X$, a sequence $(r_i)_{i=0}^\infty$ decreasing to 0, and a nonprincipal ultrafilter $\U$ on $\N$, we define the \emph{tangent cone} of $X$ at $x$, $T_x^{r_i,\U}X$, to be the ultralimit of the sequence of pointed spaces $(X,x,\frac{1}{r_i}d)$. Given a Lipschitz map $X \to B$, the blowup of $f$ at $x$, $f_x: T_x^{r_i,\U}X$, is the ultralimit of the sequence of maps $\frac{1}{r_i}(f-f(x)) + f(x): (X,x,\frac{1}{r_i}d) \to B$, if it exists. The ultralimit exists if the limit exists in the usual sense or if $B$ is finite dimensional.

A finite, metric \emph{graph} (or just graph) is a metric space $X$ equipped with a finite set of vertices, $V(X)$, and a finite set of edges, $E(X)$, satisfying some properties.
\begin{itemize}
\item $V(X) \sbs X$, and $E(X) \sbs \mathcal{P}(X)$, the power set of $X$.
\item Each $e \in E(X)$ is isometric to a compact interval $[a,b]$, and under any isometry $[a,b] \to e$, $a$ and $b$ get mapped to vertices, called the vertices of $e$, and no other point $c \in (a,b)$ gets mapped to a vertex.
\item If $e_1,e_2 \in E(X)$ with $e_1 \neq e_2$, then $e_1 \cap e_2$ is empty, or $e_1 \cap e_2$ consists of one or two vertices.
\end{itemize}

The graph is \emph{directed} if each edge is equipped with a direction, which is simply an ordering of its two vertices. The first vertex is called the \emph{source}, and the second is called the \emph{sink}. We say that the edge is directed from the source to the sink.

If $A$ is a Borel subset of a finite graph, $|A|$ denotes its length measure. If $x,y$ are points in a finite graph, $|x-y|$ denotes the distance between $x$ and $y$ with respect to the length metric, the metric given by the infimal length of paths between $x$ and $y$. A length minimizing path from $x$ to $y$ will be denoted $[x,y]$ (so that $|x-y| = |[x,y]|$), and is frequently referred to as a \emph{shortest path}. Since shortest paths need not be unique,  the notation ``$[x,y]$" does not unambiguously define one set, but it should be clear from context what is begin referred to. In any case, as far as this article is concerned, the nonuniqueness of shortest paths don't pose any problems.


\section{Inverse Systems of Nested Graphs}
\label{sec:axioms}
We begin this section by listing some axioms for a ``thick inverse system" of nested metric graphs, see Definition \ref{def:thickinvsys}. We introduce thick inverse systems for two reasons: one - we are able to prove our differentiation theorem, Theorem \ref{thm:mainthmssummary}, for the inverse limit of these systems, and two - we are able to prove that a thick inverse system can be found in any metric space containing a thick family of geodesics, see Theorem \ref{thm:existinvsys}.

\subsection{Axioms and Terminology}
\label{ss:axioms&term}

\begin{definition}
\label{def:thickinvsys}
An inverse system of nested metric measure directed graphs satisfying the following Axioms \ref{ax:graph1} - \ref{ax:thickcircsetlength} and equipped with the measure from Definition \ref{def:meas} will be called a \textbf{thick inverse system}.
\end{definition}

We use the notation $(X_0,d,\mu_0) \overset{\leftarrow}{\subseteq} (X_1,d,\mu_1) \overset{\leftarrow}{\subseteq} \dots$ for a system of nested metric directed graphs. The maps $X_{i+1} \to X_i$ are denoted $\pi^{i+1}_i$. Let $i \geq 0$ and $j \geq i$, and define $\boldsymbol{\pi^j_i} := \pi^{i+1}_i \comp \pi_{i+1}^{i+2} \comp \dots \pi^j_{j-1}: X_j \to X_i$. \\

\noindent Graph and Length Axioms:

\begin{figure}
\includegraphics[scale=.375]{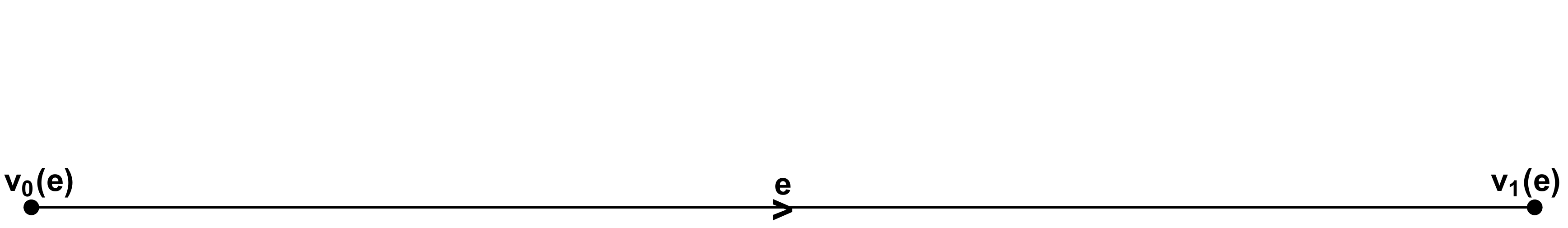}
\caption{A directed edge $e$ of $X_i$, shown in black.}
\label{fig:ei}
\includegraphics[scale=.375]{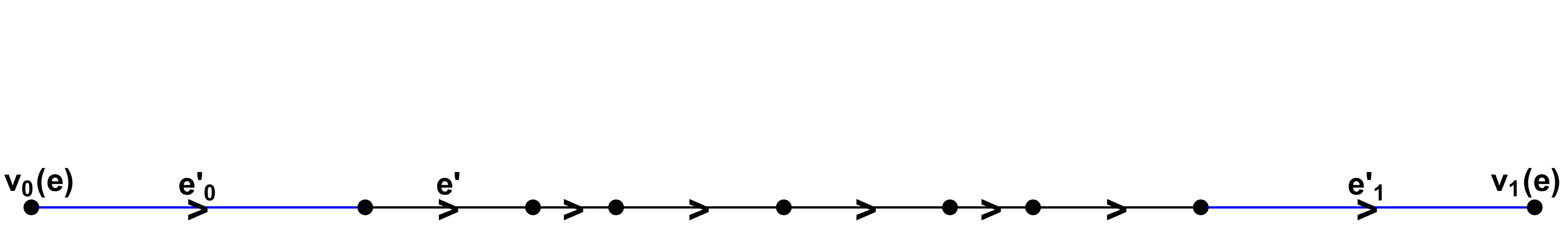}
\caption{The directed subdivision of $e$ in $X_i'$. Terminal subedges $e_0'$ and $e_1'$ are shown in blue, and nonterminal subedges are shown in black.}
\label{fig:eiprime}
\includegraphics[scale=.375]{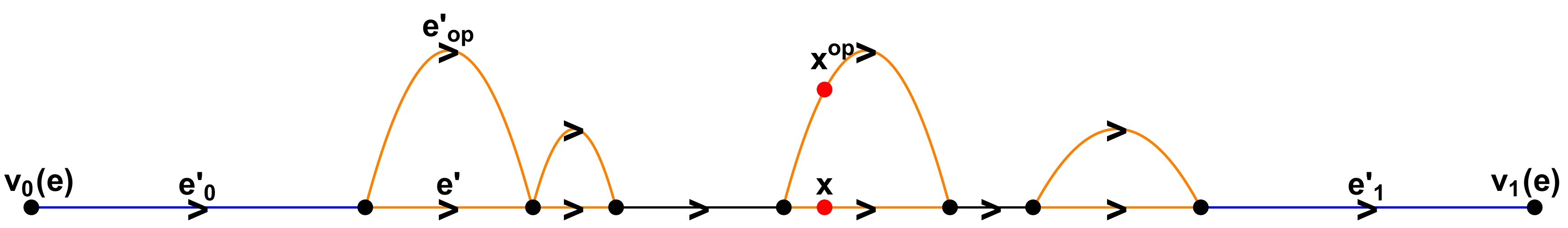}
\caption{The set $(\pi_i^{i+1})^{-1}(e)$ in $X_{i+1}$. Terminal intervals are shown in blue, nonterminal intervals are shown in black, and circles are shown in orange. Examples of subedges $e'$ and opposite subedge $e'_\op$ are labeled, as are example point and opposite point $x$ and $x^\op$, shown in red.}
\label{fig:ei+1}
\end{figure}

\begin{enumerate}[label=(A\arabic*)]
\item \label{ax:graph1} $X_0$ has two vertices, denoted 0 and 1, and one edge directed from 0 to 1, with length 1. We identify $X_0$ with $I := [0,1]$.
\item \label{ax:graph2} There is a directed subdivision of $X_i$, denoted $X_i'$, satisfying the properties below. It will be helpful to refer to Figures \ref{fig:ei}, \ref{fig:eiprime}, and \ref{fig:ei+1} while reading \ref{ax:graph2}.  
	\begin{enumerate}[label=(\roman*)]
    \item \label{subax:graph2opintcirc} For each edge $e' \in E(X_i')$, $(\pi^{i+1}_i)^{-1}(e') = e' \cup e'_{\text{op}}$, where either $e'_\text{op} = e'$, or $e'_\op$ is an edge having the same source and sink vertices as $e'$, but whose interior is disjoint from the rest of $X_{i+1}$. $e'_\text{op}$ is called the \textbf{opposite edge} of $e'$ in $X_{i+1}$ (we may write $e_\text{op}$ or $e^\text{op}$ depending on the presence of other super or subscripts). We also define $(e'_\op)_\op := e'$. \\
    
For future use, we note that, with respect to the length metric, the diameter of $(\pi^{i+1}_i)^{-1}(e')$ equals $|e'|$. Thus, with respect to $d$,
\begin{equation}
\label{eq:edgediam}
\text{diam}((\pi^{i+1}_i)^{-1}(e')) \leq |e'|
\end{equation}

Given a point $x \in e'$, we similarly define $x^\text{op}$ to be the unique point of $e'_\text{op}$ for which $\pi_i^{i+1}(x^\text{op}) = x$, and call $x^\text{op}$ the \textbf{opposite point} of $x$ (in $X_{i+1}$). (If $e'_\text{op} = e'$, then $x^\text{op} = x$. We may also write $x_\text{op}$ depending on the presence of other super or subscripts.) Again, we also define $(x^\op)^\op := x$. 

If $e'_\op = e'$, we call $(\pi^{i+1}_i)^{-1}(e')$ an \textbf{interval} (it is an interval topologically). If $e'_\op \neq e'$, we call $(\pi^{i+1}_i)^{-1}(e')$ a \textbf{circle} (it is a circle topologically).
    \item \label{subax:graph2termint}  If $e_0'$ and $e_1'$ are terminal edges in the subdivision of some edge $e \in E(X_i)$ (meaning they share a vertex with $e$), then $(\pi^{i+1}_i)^{-1}(e_0')$ and $(\pi^{i+1}_i)^{-1}(e_1')$ are intervals (so not circles). We refer to these edges of $X_i'$ and $X_{i+1}$ as \textbf{terminal intervals} (sometimes terminal edges) of $X_{i+1}$ and also as \textbf{terminal subintervals} (sometimes terminal subedges) of $e$. We note that a subedge $e' \in E(X_i')$ of $e$ is not a terminal subinterval if and only if it is contained in the interior of $e$.
    \end{enumerate}
\end{enumerate}

\noindent Metric Axioms:

\begin{enumerate}[resume*]
\item \label{ax:metricd} For any $i \geq 0$, $d$ is geodesic when restricted to any directed edge path of $X_i$, meaning there is an isometry from a compact interval to this edge path.
\item \label{ax:metricpi} $\pi_i^{i+1}: X_{i+1} \to X_i'$ acts identically on any $e' \in E(X_i') \sbs E(X_{i+1})$, and it collapses any $e'_\op \in E(X_{i+1}) \setminus E(X_i')$ isometrically onto $e'$.
\end{enumerate}

\noindent Thickness Axiom:

Suppose $e' \in E(X_i')$ is an edge such that $(\pi_i^{i+1})^{-1}(e')$ is a circle. For any $t \in e'$, let $t^\op \in e'_\op$ denote the opposite point. Define the \textbf{height} of $e'$ by \textbf{ht(}$\boldsymbol{e')} := \max_{p \in e'} d(p,p^\op)$ (the height is between 0 and $|e'|$ and is a measure of how close the circle is to being to a standard circle; it equals $|e'|$ if and only if the circle is isometric to a standard circle of diameter $|e'|$).

\begin{enumerate}[resume*]
\item \label{ax:thickcircht} There is a constant $\alpha > 0$ (independent of $i$) such that $\height{e'} \geq \alpha|e'|$ for every $e' \in E(X_{i}')$.
\item \label{ax:thickcircsetlength} Let $P$ be a directed edge path from 0 to 1 in $X_i'$, and let $E_{\text{circ}}(P)$ denote the set of edges $e' \sbs P$ along the path for which $(\pi_i^{i+1})^{-1}(e')$ is a circle. Then there is a constant $\beta > 0$ (independent of $i$ and $P$) such that $|\cup E_{\text{circ}}(P)| \geq \beta$.
\end{enumerate}

\noindent Measure Definition:

\begin{definition}
\label{def:meas}
Define $\boldsymbol{(\mu_i)_{i=0}^\infty}$ to be the unique sequence of probability measures satisfying the following recursion: $\mu_0$ is length (Lebesgue) measure on $X_0 = [0,1]$. Restricted to any edge of $X_{i+1}$, $\mu_{i+1}$ is a constant multiple of length measure and for any $e' \in E(X_i')$, $\mu_{i+1}(e') = \mu_i(e')$ if $e'_\op = e'$, and $\mu_{i+1}(e') = \mu_{i+1}(e'_\op) = \frac{1}{2}\mu_{i}(e')$ if $e'_\op \neq e'$.
\end{definition}

\subsection{Elementary Consequences of Axioms}
Throughout this subsection, fix a thick inverse system, using the same notation as in the previous subsection. We begin with a proposition that lists, without proof, some elementary consequences of the axioms. We use these facts often and without mention. Then we prove some less immediate facts about the metric structure that will be needed for subsequent results.

\begin{prop}
The following are true: 
\begin{itemize}
\item The map $\pi^{i+1}_i: X_{i+1} \to X_i$ is a projection onto $X_i \sbs X_{i+1}$; $\pi_i^{i+1}\res{X_i} = id_{X_i}$.
\item $\pi^{i+1}_i$ is direction preserving, and by induction the same is true for $\pi^j_i$, $j \geq i$. Thus, $\pi^{j}_i$ restricted to any directed edge path is an isometry.
\item The restriction of $\mu_i$ to any interval or circle of $X_i$ is a constant multiple (with constant $\leq 1$) of length measure.
\item $(\pi_i^{i+1})_\#(\mu_{i+1}) = \mu_i$.
\end{itemize}
\end{prop}

\begin{definition} \label{def:deltaE}
For any $i \geq 0$, define $\Delta_i^E := \min_{e \in E(X_i)} |e|$ and \\
$\boldsymbol{\delta^E_i} := \max_{e' \in E(X_i')}\frac{|e'|}{\Delta_i^E}$. The maximum is well-defined because each graph has finitely many edges.
\end{definition}

\begin{definition} \label{def:edges}
For any $j \geq i \geq 0$ and $x \in X_j$, define $\boldsymbol{e_i(x)}$ and $\boldsymbol{e_i'(x)}$ to be edges of $X_i$ and $X_i'$, respectively, containing $\pi^j_i(x)$. These edges are unique except when $x$ is a vertex, and the set of vertices form a measure 0 set.
\end{definition}

\begin{lemma} \label{lem:fiberdiam}
If $(\delta_i')_{i=0}^\infty$ is a positive decreasing sequence with $\delta_0' \leq \frac{1}{2}$, and if $\delta_i^E \leq \delta_i'$, then for any $j \geq i \geq 0$ and $x_i \in X_i$,
$\text{diam}((\pi_i^j)^{-1}(x_i)) \leq 2\delta_i'|e_i(x_i)|$.
\end{lemma}

\begin{proof}
Assume $\delta_i'$, $\delta^E_i$, and $x_i$ are as above. Let $x_j \in (\pi_i^j)^{-1}(x_i)$, and for $k=i \dots j$, set $x_k := \pi^j_k(x_j)$. By \eqref{eq:edgediam}, $d(x_k,x_{k+1}) \leq |e_k'(x_k)|$. By a repeated application of the definition of $\delta^E_i$, we have $|e_k'(x_k)| \leq \delta^E_i \cdot \delta^E_{i+1} \cdot \dots \delta^E_{k} |e_i(x_i)| \leq (\delta_i')^{k+1-i} |e_i(x_i)|$, where the least inequality holds since $\delta^E_k \leq \delta_k'$ and $\delta_k'$ is decreasing. Then we have $d(x_i,x_j) \leq \sum_{k=i}^{j-1} d(x_k,x_{k+1}) \leq \left( \sum_{k=i}^{j-1} (\delta_i')^{k+1-i} \right)|e_i(x_i)| \leq 2\delta_i'|e_i(x_i)|$, where the last inequality holds since $\delta_i' \leq \delta_0' \leq \frac{1}{2}$.
\end{proof}

\begin{definition} \label{def:deltad}
For any $i \geq 0$, $e \in E(X_i)$ and $e' \in E(X_i')$ with $e'$ a nonterminal subedge of $e$, define $\Delta_i^d(e,e') := d(e',X_i \setminus e)$. This is positive by compactness and since $e'$ belongs to the interior of $e$ (since it is nonterminal). Define $\Delta_i^d(e)$ to be the minimum of $\Delta_i^d(e,e')$ over all nonterminal $e' \sbs e$, and define $\Delta_i^d$ to be the minimum of $\Delta_i^d(e)$ over all $e \in E(X_i)$. Define $\boldsymbol{\delta_i^d} := \max_{e'} \frac{|e'|}{\Delta_i^d}$, where the max is over all nonterminal edges $e' \in E(X_i')$.
\end{definition}

\begin{lemma} \label{lem:piLipbnd}
If $\delta^d_i < \frac{1}{2}$ and $\Pi_{i=0}^\infty \frac{1}{1-2\delta^d_i} \leq L$, then $\Lip{\pi^{j}_i} \leq \Pi_{k=i}^{j-1} \Lip{\pi^{k+1}_k} \leq \Pi_{i=0}^\infty \frac{1}{1-2\delta^d_i} \leq L$ for any $j \geq i \geq 0$.
\end{lemma}

\begin{proof}
It suffices to prove $\Lip{\pi^{k+1}_k} \leq \frac{1}{1-2\delta^d_k}$. Let $x_{k+1},y_{k+1} \in X_{k+1}$, and set $x_k := \pi^{k+1}_k(x_{k+1})$, $y_k := \pi^{k+1}_k(y_{k+1})$. We need to show that $d(x_{k+1},y_{k+1}) \geq (1-2\delta_k^d)d(x_{k},y_{k})$. We consider two cases; either $x_k$ and $y_k$ belong to the same edge of $X_k$, or they belong to different edges. Assume they belong to the same edge. Then there are again two cases; either $x_{k+1}$ and $y_{k+1}$ belong to opposite edges of a circle, or they belong to a directed edge path. The conclusion holds in this second case since the map $\pi_k^{k+1}$ is an isometry on directed edges paths (so we get an ever better bound of 1). Now suppose they belong to opposite edges of a circle. Without loss of generality, assume $y_{k+1} \in e'$ and $x_{k+1} \in e'_\op$ for some $e' \in E(X_k')$. Then $y_{k+1} = y_k$, and a shortest path between them, $[x_{k+1},y_{k+1}]$, passes through one of the vertices of the circle, say $v$. Then since $x_k$ and $v$ belong to an edge, $d(x_k,v) = |v-x_k|$ (recall that $|p-q|$ denotes the distance with respect to the length metric), and since $v$ and $y_k$ belong to an edge, so $d(v,y_k) = |y_k-v|$. Without loss of generality, assume $|v-x_k| \leq |y_k-v|$. This implies $x_k \in [v,y_k]$, in turn implying $|x_k-v| + |y_k-x_k| = |y_k-v|$. Then we have
$$d(x_{k+1},y_{k+1}) \geq d(v,y_{k+1}) - d(x_{k+1},v) = d(v,y_k) - d(v,x_k)$$
$$= |y_k-v| - |x_k-v| = |y_k-x_k| = d(y_k,x_k)$$ 
Our conclusion holds in this case (again with an ever better bound of 1).

Finally, assume that $x_{k}$ and $y_{k}$ do not belong to the same edge of $X_k$. We consider three cases now: both points belong to a terminal interval of $X_k$, neither point does, or one does and the other does not. Our conclusion holds the first case, since $\pi_{k}^{k+1}$ acts identically on $X_k$ (so $y_{k+1} = y_k$ and $x_{k+1} = x_k$), and terminal intervals belong to $X_k$ by definition. Assume the second case holds. $e_k'(x_k)$ and $e_k'(y_k)$ are nonterminal by assumption. Then by definition of $\delta_k^d$, since $y_k$ and $x_k$ do not belong to the same edge of $X_k$, $|e_k'(y_k)|, |e_k'(x_k)| \leq \delta_k^d d(x_k,y_k)$. Then we have
$$d(x_{k+1},y_{k+1}) \geq d(x_k,y_k) - d(x_{k+1},x_k) - d(y_{k+1},y_k) \geq d(x_k,y_k) - |e_k'(x_k)| - |e_k'(y_k)|$$
$$\geq d(x_k,y_k) - 2\delta_k^d d(x_k,y_k) = (1-2\delta_k^d)d(x_{k},y_{k})$$
And our desired conclusion holds in this case. For the third and final case, assume without loss of generality that $y_k$ belongs to a terminal interval and $x_k$ does not. Then we get $y_{k+1} = y_k$ and $|e_k'(x_k)| \leq \delta_k^d d(x_k,y_k)$. Making the obvious adjustments to the argument above yields
$$d(x_{k+1},y_{k+1}) = d(x_{k+1},y_k) \geq d(x_k,y_k) - d(x_{k+1},x_k) \geq d(x_k,y_k) - |e_k'(x_k)|$$
$$\geq d(x_k,y_k) - \delta_k^d d(x_k,y_k) = (1-\delta_k^d)d(x_{k},y_{k})$$
\end{proof}

\begin{lemma} \label{lem:pi-1Lipbnd}
If $x_{k+1},y_{k+1} \in X_{k+1}$ do not belong to opposite open edges of a circle, then $d(\pi^{k+1}_k(x_{k+1})),\pi^{k+1}_k(y_{k+1})) \geq \frac{1}{1+2\delta_k^d}d(x_{k+1},y_{k+1})$ (loosely, $\pi^{k+1}_k$ collapses circles, but is close to an isometry away from them).
\end{lemma}

\begin{proof}
Let $x_{k+1},y_{k+1} \in X_{k+1}$, and set $x_k := \pi^{k+1}_k(x_{k+1})$, $y_k := \pi^{k+1}_k(y_{k+1})$. As before, there are two cases; either $x_k$ and $y_k$ belong to the same edge of $X_k$, or they belong to different edges. Assume they belong to the same edge. Again, as before, there are two cases; either $x_{k+1}$ and $y_{k+1}$ belong to opposite edges of a circle, or they belong to a directed edge path. The first case doesn't hold by assumption, and the conclusion holds in this second case since the map $\pi_k^{k+1}$ is an isometry on directed edges paths (so we get an ever better bound of 1).

Finally, assume that $x_{k}$ and $y_{k}$ do not belong to the same edge of $X_k$. As before, three cases: both points belong to a terminal interval of $X_k$, neither point does, or one does and the other does not. Our conclusion holds the first case, since $\pi_{k}^{k+1}$ acts identically on $X_k$ (so $y_{k+1} = y_k$ and $x_{k+1} = x_k$, and intervals belong to $X_k$ by definition. Assume the second case holds. $e_k'(x_k)$ and $e_k'(y_k)$ are nonterminal by assumption. Then by definition of $\delta_k^d$, since $y_k$ and $x_k$ do not belong to the same edge of $X_k$, $|e_k'(y_k)|, |e_k'(x_k)| \leq \delta_k^d d(x_k,y_k)$. Then we have
$$d(x_{k+1},y_{k+1}) \leq d(x_k,x_{k+1}) + d(x_k,y_k) + d(y_k,y_{k+1})$$
$$\leq |e_k'(x_k)| + d(x_k,y_k) + |e_k'(y_k)| \leq (1+2\delta_k^d)d(x_k,y_k)$$
And our desired conclusion holds in this case. For the third and final case, assume without loss of generality that $y_k$ belongs to a terminal interval and $x_k$ does not. Then we get $y_{k+1} = y_k$ and $|e_k'(x_k)| \leq \delta_k^d d(x_k,y_k)$. Making the obvious adjustments to the argument above yields
$$d(x_{k+1},y_{k+1}) = d(x_{k+1},y_k) \leq d(x_k,x_{k+1}) + d(x_k,y_k)$$
$$\leq |e_k'(x_k)| + d(x_k,y_k) \leq (1+\delta_k^d)d(x_k,y_k)$$
\end{proof}

\begin{remark}
Note that since $\frac{1}{1-2\delta} > 1+\delta$, if the hypotheses of Lemma \ref{lem:pi-1Lipbnd} are satisfied, then
\begin{equation} \label{eq:pi-1Lipbndprod}
\Pi_{k=0}^\infty (1+\delta_k^d) \leq L
\end{equation}
\end{remark}

\subsection{Existence of Inverse System}
Let $M$ be a metric space.

\begin{theorem} \label{thm:existinvsys}
If $M$ contains a thick family of geodesics, then for any positive sequence $(\delta_i')_{i=0}^\infty$, $M$ contains a thick inverse system with $\delta_i^E,\delta_i^d \leq \delta_i'$ for every $i$ (see Definitions \ref{def:thickinvsys}, \ref{def:deltaE}, and \ref{def:deltad}).
\end{theorem}

\begin{proof}
Assume $M$ contains an $\alpha'$-thick family of geodesics $\Gamma$ for some $\alpha' > 0$. Let $(\delta_i')_{i=0}^\infty$ be a positive sequence. We'll construct the inverse sequence $X_0 \overset{\leftarrow}{\sbs} X_1 \overset{\leftarrow}{\sbs} \dots$ inductively. Let $\gamma$ be any element of $\Gamma$, and set $X_0$ equal to the image of $\gamma$ in $M$. Equip $X_0$ with the necessary graph structure. Assume $X_i$, $X_{i-1}$, and $\pi^{i}_{i-1}$ have been constructed for some $i \geq 0$, satisfy the Graph, Metric, and Thickness Axioms, and also satisfy the additional hypothesis that the geodesic parametrization of each directed 0-1 edge path belongs to $\Gamma$. For each edge $e \in E(X_i)$, let $v_0(e)$ and $v_1(e)$ denote the source and sink vertices of $e$, respectively. $e$ is mapped isometrically onto $I_e := [d(0,v_0(e)),d(v_1(e),1)]$ via $\pi^i_0$. Denote the inverse of this map $\gamma_e: I_e \to e \sbs X_i$. Note that, for any geodesic parametrization $\gamma$ of a 0-1 edge path whose image contains $e$, we must have $\gamma\res{I_e} = \gamma_e$, so $\gamma_e$ extends to a geodesic parametrization of a directed 0-1 edge path.

Now we provide a more quantitative reformulation of Definition \ref{def:thickfamily}. By a \emph{partition} $T$ of an interval $[a,b]$, we mean a finite subset of $[a,b]$ equipped with the order induced from $[a,b]$, such that the least element is $a$ and the greatest element is $b$. For any $t \in T$ other than $b$, we define $t^+$ to be the immediate successor of $t$, and we simply define $b^+ := b$, and for any $t \in T$ other than $a$, we define $t^-$ to be the immediate predecessor of $t$, and we simply define $a^- := a$. For each partition $T^e$ of $I_e$, and $\tilde{\gamma}^e \in \Gamma$ with $\tilde{\gamma}^e\res{T^e} \equiv \gamma_e\res{T^e}$ define the \emph{deviations} of $(T^e,\tilde{\gamma}^e)$ and $T^e$, respectively:
$$\dev{T^e,\tilde{\gamma}^e} := \sum_{t \in T^e} \max_{s \in [t,t^+]} d(\gamma_e(s),\tilde{\gamma}^e(s))$$
$$\dev{T^e} := \sup_{\substack{\tilde{\gamma}^e \in \Gamma \\ \tilde{\gamma}^e\res{T^e} \equiv \gamma_e\res{T^e}}} \dev{T^e,\tilde{\gamma}^e}$$
Note that $\dev{T^e} \leq |e|$.

For any fixed partition $T^e$ of $I_e$, let $T^e_{sup/2}$ and $\tilde{\gamma}^e_{sup/2}$ denote a partition of $I_e$ and a geodesic in $\Gamma$, respectively, with
\begin{equation} \label{eq:sup/2prop1}
T^e_{sup/2} \sps T^e
\end{equation}
\begin{equation} \label{eq:sup/2prop2}
\tilde{\gamma}^e_{sup/2}\res{T^e_{sup/2}} \equiv \gamma_e\res{T^e_{sup/2}}
\end{equation}
\begin{equation} \label{eq:sup/2prop3}
\dev{T^e_{sup/2},\tilde{\gamma}^e_{sup/2}} \geq \frac{1}{2} \sup_{T' \sps T^e} \dev{T'}
\end{equation}

Now, we can always choose $T^e_{sup/2}$ and $\tilde{\gamma}^e_{sup/2}$ such that the above properties remain true, and also such that for every $t \in T^e_{sup/2}$,
\begin{equation} \label{eq:sup/2prop4}
\gamma_e\res{[t,t^+]} \equiv \tilde{\gamma}^e_{sup/2}\res{[t,t^+]} \:\:\:\: \text{or} \:\:\:\: \gamma_e((t,t^+)) \cap \tilde{\gamma}^e_{sup/2}((t,t^+)) = \emptyset
\end{equation}
To see this, take any $T^e_{sup/2}$ and $\tilde{\gamma}^e_{sup/2}$ as above, and let $t \in T^e_{sup/2}$. If \\
$\max_{s \in [t,t^+]} d(\gamma_e(s),\tilde{\gamma}^e(s)) = 0$, then $\gamma_e$ and $\tilde{\gamma}^e_{sup/2}$ agree on all of $[t,t^+]$ and we are done. Otherwise, let $s_{max} = \text{argmax}_{s \in [t,t^+]} d(\gamma_e(s),\tilde{\gamma}^e(s))$. Then by continuity, there exists a largest, nonempty open subinterval $(a,b)$ of $[t,t^+]$ containing $s_{max}$ such that $\gamma_e((a,b)) \cap \tilde{\gamma}^e_{sup/2}((a,b)) = \emptyset$. Since it is the largest, $\gamma_e(a) = \tilde{\gamma}^e_{sup/2}(a)$ and $\gamma_e(b) = \tilde{\gamma}^e_{sup/2}(b)$. We add these new points $a$ and $b$ to the partition $T^e_{sup/2}$, and modify $\tilde{\gamma}^e_{sup/2}$ so that it agrees with $\gamma_e$ on $[t,a] \cup [b,t^+]$, and remains unchanged on $[a,b]$. This new curve still belongs to $\Gamma$ because $\Gamma$ is concatenation closed. It is clear that \eqref{eq:sup/2prop1}, \eqref{eq:sup/2prop2}, and \eqref{eq:sup/2prop3} remain valid, and that we gain \eqref{eq:sup/2prop4}.

We use the partition $T^e_{sup/2}$ of $I_e$ to subdivide $e$ into smaller edges by taking the image of $T^e_{sup/2}$ under $\gamma_e$ to be new vertices. Each new subedge equals $\gamma_e([t,t^+])$ for a unique $t \in T^e_{sup/2}$. Denote this edge $e^t$, and recall the \emph{height} of $e^t$, defined in the Thickness Axioms,
$$\height{e^t} = \max_{s \in [t,t^+]} d(\gamma_e(s),\tilde{\gamma}^e_{sup/2}(s))$$
Set $\alpha := \frac{\alpha'}{4}$, and split the new subedges of $e$ up into two groups, $E^e_{< \alpha}$ and $E^e_{\geq \alpha}$, where $e^t \sbs e$ belongs to $E^e_{< \alpha}$ if $\height{e^t} < \alpha|e^t|$ and $e^t$ belongs to $E^e_{\geq \alpha}$ if $\height{e^t} \geq \alpha|e^t|$. Name the collection of corresponding time intervals $(T^e_{sup/2})_{< \alpha}$ and $(T^e_{sup/2})_{\geq \alpha}$.

It follows from Definition \ref{def:thickfamily} and the observation that $\gamma^e$ extends to a geodesic in $\Gamma$, that for any 0-1 directed edge path $P$, and any choice of partition $T^e$ for each $e \sbs P$,
$$\sum_{e \sbs P} \dev{T^e_{sup/2},\tilde{\gamma}^e_{sup/2}}  \geq \frac{\alpha'}{2}$$
It follows from this that
$$\frac{\alpha'}{2} \leq \sum_{e \sbs P} \left( \sum_{e^t \in E^e_{< \alpha}} \height{e^t} + \sum_{e^t \in E^e_{\geq \alpha}} \height{e^t} \right)$$
$$\leq \sum_{e \sbs P} \left( \sum_{e^t \in E^e_{< \alpha}} \alpha|e^t| + \sum_{e^t \in E^e_{\geq \alpha}} |e^t| \right)$$
$$\leq \sum_{e \sbs P} \left( \alpha|e| + \sum_{e^t \in E^e_{\geq \alpha}} |e^t| \right)$$
$$= \alpha + \sum_{e \sbs P}\sum_{e^t \in E^e_{\geq \alpha}} |e^t|$$
implying
\begin{equation} \label{eq:circsetlength}
|\cup E^P_{\geq \alpha}| \geq 2\beta
\end{equation}
where $\beta := \frac{\alpha'}{8}$ and $E^P_{\geq \alpha} = \cup_{e \sbs P} E^e_{\geq \alpha}$. Now that the preliminaries have been established, we are ready to choose a specific partition of $e$ and apply the above results.

Set $\Delta_i^E := \min_{e \in E(X_i)} |e|$, and for each $e \in E(X_i)$, subdivide $e$ into three edges $e_0' < e_{mid} < e_1'$ such that $|e_0'| = \min(\frac{\beta}{2}|e|,\delta_i'\Delta_i^E) = |e_1'|$. Set $\Delta_i^d(e) := d(e_{mid},X_i \setminus e)$. Since $e_{mid}$ belongs to the interior of $e$, compactness gives us $\Delta_i^d(e) > 0$. Then set $\Delta_i^d := \min_{e \in E(X_i)} \Delta_i^d(e)$ and $\eps_i := \min(\delta_i'\Delta_i^E,\delta_i'\Delta_i^d)$. Now, for each $e \in E(X_i)$, choose a partition $T^e$ of $I_e = [a,b] = [d(0,v_0(e)),d(v_1(e),1)]$ such that
\begin{equation} \label{eq:termsubedgelength}
a^+ - a =  \min\left(\frac{\beta}{2}|e|,\delta_i'\Delta_i^E\right) = b - b^-
\end{equation}
(this implies $\gamma_e([a,a^+]) = e_0'$, $\gamma_e([b^-,b]) = e_1'$, and $\gamma_e([t,t^+]) \sbs e_{mid}$ for $t \in T^e \setminus \{a,b^-,b\}$) and for any $t \in T^e \setminus \{a,b^-,b\}$
\begin{equation} \label{eq:nontermsubedgelength}
t^+ - t \leq \eps_i
\end{equation}

For each $e \in E(X_i)$, fix $T^e_{sup/2} \sps T^e$ and $\tilde{\gamma}^e_{sup/2}$ as before. As explained in the previous paragraph, $T^e_{sup/2}$ induces a subdivision of $e$. Doing this for each $e$ gives us the total subdivided graph $X_i'$. By \eqref{eq:termsubedgelength} and \eqref{eq:nontermsubedgelength}, any subedge $e^t \sbs e$ satisfies $|e^t| \leq \delta_i' \Delta_i^E$, so $\delta_i^E \leq \delta_i'$, as required. Furthermore, any nonterminal subedge $e^t$ of $e$ is contained in $e_{mid}$, by definition, and so by \eqref{eq:nontermsubedgelength} we get $|e^t| \leq \eps_i \leq \delta_i' \Delta_i^d$, implying $\delta_i^d \leq \delta_i'$, as required.

It remains to construct $X_{i+1}$ and $\pi_i^{i+1}$. We explain how to use segments of the curve $\tilde{\gamma}^e_{sup/2}$ as new edges to add to our graph $X_i'$ to obtain $X_{i+1}$. Let $e \in E(X_i)$. There are three options for a subedge $e' \in E(X_i')$ of $e$: $e'$ is a terminal subedge, (meaning $e' = e_0' = e^t$ or $e_1'^t$ for $t \in \{d(0,v_0(e)),d(v_1(e),1)^-\}$), $e' = e^t$ for some $t \in (T^e_{sup/2})_{\geq \alpha} \setminus \{d(0,v_0(e)),d(v_1(e),1)^-\}$ (meaning $\height{e^t} \geq \alpha|e^t|$), or $e' = e^t$ for some $t \in (T^e_{sup/2})_{< \alpha} \setminus \{d(0,v_0(e)),d(v_1(e),1)^-\}$ (meaning $\height{e^t} < \alpha|e^t|$). In the first two cases, we set $e'_\op = e'$, so that $(\pi^{i+1}_i)^{-1}(e')$ is a circle, and in the third case, set $e'_\op = e^t_\op := \tilde{\gamma}^e_{sup/2}([t,t^+])$, so that the intersection of the interiors of $e^t$ and $e^t_\op$ is empty, $(\pi^{i+1}_i)^{-1}(e')$ is a circle, and $\height{e^t} \geq \alpha|e^t|$. We define $\pi^{i+1}_i$ in the unique way so that \ref{ax:metricpi} holds. It is clear that the Graph Axioms, Metric Axioms, and \ref{ax:thickcircht} hold. Our additional hypothesis that the geodesic parametrization of every 0-1 directed edge path belongs to $\Gamma$ also holds (again using concatenation closed). It remains to verify Axiom \ref{ax:thickcircsetlength}. 

To verify \ref{ax:thickcircsetlength}, we fix a path $P$ and compute $|\cup E_{\text{circ}}(P)|$. For each $e \sbs P$, set $E_{\text{circ}}(e) = \{e' \in E_{\text{circ}}(P): e' \sbs e\}$. Then by \eqref{eq:circsetlength} and \eqref{eq:termsubedgelength},
$$|\cup E_{\text{circ}}(P)| = \sum_{e \sbs P} |\cup E_{\text{circ}}(e)| = \sum_{e \sbs P} \left(|\cup E^e_{\geq \alpha}| - |(\cup E^e_{\geq \alpha}) \cap (e_0' \cup e_1')|\right)$$
$$\geq \sum_{e \sbs P} \left(|\cup E^e_{\geq \alpha}| - |e_0' \cup e_1'|\right) \overset{\eqref{eq:termsubedgelength}}{\geq} \sum_{e \sbs P} \left(|\cup E^e_{\geq \alpha}| - \left(\frac{\beta}{2}|e| + \frac{\beta}{2}|e|\right)\right)$$
$$= |\cup E^P_{\geq \alpha}| - \beta \overset{\eqref{eq:circsetlength}}{\geq} 2 \beta - \beta = \beta$$
\end{proof}

From here till the end of Section \ref{sec:weakRNPdiff}, fix a complete metric space $(M,d)$ containing a thick family of geodesics, a positive sequence $(\delta_i')_{i=0}^\infty$ decreasing to 0 quickly enough so that $\delta_0' < \frac{1}{2}$ and $\Pi_{i=0}^\infty \frac{1}{1-2\delta_i'} \leq L$ for some $L < \infty$ (this also implies $\sum_i \delta_i' < \infty)$, and a thick inverse system afforded to us by the theorem.

\begin{definition}
Denote the closure of $X_{< \infty} := \cup_{i=0}^\infty X_i$ inside $M$ as $\boldsymbol{X_\infty}$. We fix $0 \in I = X_0 \sbs X_\infty$ to be the basepoint. By Lemma \ref{lem:piLipbnd}, the maps $\pi^j_i$ are uniformly $L$-Lipschitz, so we get $L$-Lipschitz extensions $\boldsymbol{\pi^\infty_i}: X_\infty \to X_i$. Summarizing:
\begin{equation} \label{eq:piLipbnd}
\forall j \in \{i,i+1, \dots \infty\}, \:\:\: \Lip{\pi^j_i} \leq L
\end{equation}

We also extend the definitions of $e_i(x)$ and $e_i'(x)$ (see Definition \ref{def:edges}) in the obvious way when $x \in X_\infty$.
\end{definition}

\begin{remark}
By Lemma \ref{lem:fiberdiam}, we get
\begin{equation} \label{eq:fiberdiam}
\forall j \in \{i,i+1, \dots \infty\}, \: x_i \in X_i, \:\:\: \text{diam}((\pi_i^j)^{-1}(x_i)) \leq 2\delta_i'|e_i(x_i)| 
\end{equation}

Since each $(X_i,d)$ is a finite graph, each $(X_i,d)$ is compact and thus totally bounded. Then \eqref{eq:fiberdiam}, together with our choice that $\delta_i' \to 0$, imply $X_{<\infty}$ is totally bounded. Then since $M$ is complete, $X_\infty$ is compact. 

The maps $\pi^\infty_i: X_\infty \to X_i$ are each $L$-Lipschitz and act identically on $X_i \sbs X_\infty$. These two facts imply, for any $p,q \in X_\infty$,
\begin{equation} \label{eq:dconv}
d(p,q) = \lim_{i \to \infty}d(\pi^\infty_i(p),\pi^\infty_i(q))
\end{equation}

This implies that the maps $\pi^\infty_i$ generate the topology on $X_\infty$, i.e., the topology on $X_\infty$ is the weakest one such that each map $\pi^\infty_i$ is continuous. Equivalently, the subalgebra of $C(X_\infty)$ consisting of those continuous functions that factor through some $\pi^\infty_i$ is dense. We denote this subalgebra by $C_{\text{unif}}(X_{< \infty})$. The compatibility condition of the probability measures ($(\pi^{i+1}_i)_\#(\mu_{i+1}) = \mu_i$) gives us a well-defined, bounded, positive linear functional $\lambda_{< \infty}$ on $C_{\text{unif}}(X_{< \infty})$. By density this extends to a unique positive linear functional $\lambda_\infty$ on all of $C(X_\infty)$.
\end{remark}

\begin{definition}
Define $\boldsymbol{\mu_\infty}$ to be the Radon measure representing the linear functional $\lambda_\infty$ on $C(X_\infty)$. $\mu_\infty$ is a probability measure uniquely characterized by:
\begin{equation} \label{eq:measpush}
\forall i \geq 0, \:\:\: (\pi^\infty_i)_\#(\mu_\infty) = \mu_i
\end{equation}
\end{definition}

\begin{remark}
Although we won't make explicit use it, we believe it is worth mentioning the following fact: the metric space $X_\infty$ and maps $(\pi^\infty_i)_{i=0}^\infty$ satisfy the universal property of an inverse limit space. This means that for any metric space $Y$ and uniformly Lipschitz sequence of maps $(f_i)_{i=0}^\infty$, $f_i: Y \to X_i$, there exists a unique Lipschitz map $f_\infty: Y \to X_\infty$ such that $\pi^\infty_i \comp f_\infty = f_i$ for any $i$.
\end{remark}

\section{Asymptotic Local Properties of $(X_i)_{i=0}^\infty$ and Special Subsets of $X_\infty$}
\label{sec:asymprops}


\subsection{Deep Points and their Natural Scales}
Recall the definition of \emph{terminal intervals} of $X_{i+1}$ from Axiom \ref{ax:graph2}\ref{subax:graph2termint}.

\begin{definition}
We define the set of \textbf{deep points}, $\boldsymbol{D}$, to be all those $x \in X_\infty$ such that $\pi^\infty_{i+1}(x)$ eventually (in $i$) does not belong to a terminal interval of $X_{i+1}$. $D$ is a $G_{\delta\sigma}$ (and hence Borel) set.
\end{definition}

\begin{theorem}
$\mu_\infty(D) = 1$.
\end{theorem}

\begin{proof}
Let $e$ be an edge of $X_i$ and $e_0'$ and $e_1'$ its terminal subintervals. By Definition \ref{def:meas} and Definition \ref{def:deltaE}, $\mu_{i+1}(e_0' \cup e_1') = \mu_{i+1}(e_0') + \mu_{i+1}(e_1') \leq 2\delta_i^E\mu_i(e)$. Summing over all $e \in E(X_i)$, we get that the total measure of the union of terminal intervals in $E(X_{i+1})$ is bounded by $2\delta_i^E$. Since $\sum_i \delta^E_i \leq \sum_i \delta_i' < \infty$, Borel-Cantelli implies that the set of $x \in X_\infty$ such that $\pi^\infty_{i+1}(x)$ eventually (in $i$) does not belong to a terminal interval in $X_{i+1}$ has measure 1.
\end{proof}


\subsubsection{Structure of $(\pi^{i}_{i-1})^{-1}(e)$}
We now discuss some geometric properties of \\
$(\pi^{i}_{i-1})^{-1}(e)$. While reading this section, it will be helpful to refer to Figure \ref{fig:ei+1} for a picture of what $(\pi^{i}_{i-1})^{-1}(e)$ typically looks like.

\begin{definition}
Given a deep point or, more generally, a nonvertex $x$ and $i \geq 0$, define $\boldsymbol{r_i(x)} := |e_i(x)|$. We call $r_i(x)$ the sequence of \emph{natural scales} of $X_\infty$ at $x$.
\end{definition}

\begin{lemma} \label{lem:ballcontainedinedge}
For any deep point $x$ and $R \geq 1$, $B^i_{Rr_i(x)}(\pi^\infty_i(x))$ is eventually (in $i$, depending on $x$ and $R$) contained in $(\pi^i_{i-1})^{-1}(e_{i-1}(x))$, where $B^i$ indicates a ball in the space $(X_i,d)$.
\end{lemma}

\begin{proof}
Let $x \in D$ and $R \geq 1$. Set $x_i := \pi^\infty_i(x)$ and assume $i$ is large enough so that $e_{i-1}'(x)$ is not a terminal interval. Then by Definition \ref{def:deltad}, $d(x_{i-1},X_{i-1} \setminus e_{i-1}(x)) \geq \frac{|e_{i-1}'(x)|}{\delta_{i-1}^d} = \frac{r_i(x)}{\delta_{i-1}^d} \geq \frac{r_i(x)}{\delta_{i-1}'}$. Combining this with \eqref{eq:piLipbnd} yields
$$d(x_i,X_i \setminus (\pi^i_{i-1})^{-1}(e_{i-1}(x))) \geq \frac{1}{L} d(x_{i-1},X_{i-1} \setminus e_{i-1}(x)) \geq \frac{r_i(x)}{L\delta_{i-1}'}$$
Thus, as soon as $i$ is large enough so that $\delta_{i-1}' < \frac{1}{LR}$, we get \\
$B^i_{Rr_i(x_i)} \sbs (\pi^i_{i-1})^{-1}(e_{i-1}(x))$.
\end{proof}

\begin{lemma} \label{lem:doubling}
\begin{enumerate}
\item There exists $C \geq 1$ such that for any $i \geq 0$ and $e \in E(X_{i-1})$, $\mu_{i}$ restricted to $(\pi^{i}_{i-1})^{-1}(e)$ is $C$-doubling with respect to the length metric.
\item For any shortest path $[x,y] \sbs (\pi^{i}_{i-1})^{-1}(e)$, $\mu_i(B_r(x)) \leq 4\mu_i([x,y])$, where $r = |x-y|$.
\end{enumerate}
\end{lemma}

\begin{proof}
Let $i \geq 0$ and $e \in E(X_{i-1})$. Recall the definition of \emph{circles} and \emph{intervals} from Axiom \ref{ax:graph2}\ref{subax:graph2opintcirc}. By the discussion there, $(\pi^{i}_{i-1})^{-1}(e) = \cup_{e' \sbs e} (\pi^{i}_{i-1})^{-1}(e')$ consists of a sequence of intervals and circles, glued together in a directed way along alternating sink and source vertices. This sequence begins and ends with \emph{terminal intervals}, defined in Axiom \ref{ax:graph2}\ref{subax:graph2termint}. With respect to the length metric and length measure, $(\pi^{i}_{i-1})^{-1}(e)$ is doubling. This follows by analyzing the worst case scenario for a ball. This scenario occurs near points where two circles are glued together. It is possible to have a geodesic ball of radius $r$ such that the geodesic ball of radius $2r$ has 4 times the length. This implies length measure is doubling with doubling constant 4. Let $c \in (0,1]$ such that $\mu_{i-1}$ restricted to $e$ equals $c$ times length measure, and for any $e \sps e' \in E(X_{i-1}')$, $\mu_{i}$ restricted to $(\pi^{i}_{i-1})^{-1}(e') \sbs (\pi^{i}_{i-1})^{-1}(e)$ equals $c$ or $\frac{c}{2}$ times length measure ($c$ if it's an interval, $\frac{c}{2}$ if it's a circle). It follows that $\mu_{i}$ restricted to $(\pi^{i}_{i-1})^{-1}(e)$ is bounded above by $c$ times length measure and below by $\frac{c}{2}$ times length measure. Since length measure it doubling with doubling constant 4, this implies $\mu_{i}$ is doubling with doubling constant bounded by 8 (this isn't sharp).

The second statement can also be observed by examining the worst case scenario where $x$ is a vertex shared by two adjacent circles and $y$ belongs to one of these circles. Then $B_r(x)$ will consist of four copies of an interval of length $r = |x-y|$, and the $\mu_i$ measure of any of these new intervals is the same as that of $[x,y]$. This implies the second statement.
\end{proof}

\begin{remark}
It's also clear from the description of $(\pi_{i-1}^i)^{-1}(e_{i-1}(x))$ given in the preceding section that if $x,y \in (\pi^{i+1}_i)^{-1}(e)$ and $x$ and $y$ do not belong to opposite edges of a circle, then $x$ and $y$ belong to a directed (and thus geodesic) edge path, and so $d(x,y) = |y-x|$. On the other hand, if $y \in B^i_{Rr_i(x)}(x_i)$ and $x_i$ and $y$ belong to opposite edges of a circle, then $|y-x_i| \leq |e_i(x)|$. In either case, we have, for $R \geq 1$ and $i$ sufficiently large,

\begin{equation} \label{eq:quasigeo}
\forall y \in B^i_{Rr_i(x)}(x_i), \:\:\: |y-x_i| \leq Rr_i(x)
\end{equation}
\end{remark}


\subsection{Points having a NonEuclidean Tangent}
\begin{theorem} \label{thm:Sinfty}
There exists a Borel $S_\infty \sbs X_\infty$ such that $\mu_\infty(S_\infty) > 0$, and for all $x \in S_\infty$, there exists a nonprincipal ultrafilter $\U(x)$ (depending on $x$) on $\N$ such that the tangent cone $T_x^{r_i(x),\U(x)}X_\infty$ does not embed (even topologically) into $\R$.
\end{theorem}

Before beginning the proof of the theorem, we require a lemma:

\begin{lemma} \label{lem:measdisint}
For each $i \geq 0$, there is a finite set of directed 0-1 edge paths of $X_i$, $\mathcal{P}_i$, and a probability measure $\P_i$ on $\mathcal{P}_i$ such that for every edge $e \in E(X_i)$,
\begin{equation*}
\frac{\mu_i(e)}{|e|} = \sum_{\substack{P \in \mathcal{P}_i \\ e \sbs P}}\P_i(P)
\end{equation*}
and it follows that, for any $A \sbs e \in E(X_i)$ Borel,
\begin{equation} \label{eq:measdisint}
\mu_i(A) = \sum_{P \in \mathcal{P}_i}\P_i(P)|A \cap P|
\end{equation} 
\end{lemma}

\begin{proof}
The proof is by induction on $i$. The base case $i = 0$ holds trivially with $\mathcal{P}_0 = \{X_0\}$, $\P_0 = \delta_{X_0}$. Assume the statement holds for some $i \geq 0$. Let $P \in \mathcal{P}_i$. Let $P_\op$ be the unique 0-1 directed edge path in $X_{i+1}$ such that $e' \sbs P$ if and only if $e'_\op \sbs P_\op$ for every $e' \in E(X_i')$. Let $\mathcal{P}_{i+1} = \{P,P_\op\}_{P \in \mathcal{P}_i}$. For each $P \in \mathcal{P}_i$, define $\P_{i+1}(P_\op): = \P_{i+1}(P): = \frac{1}{2}\P_{i}(P)$ if $P_\op \neq P$, and $\P_{i+1}(P_\op) = \P_{i+1}(P): = \P_{i}(P)$ if $P_\op = P$. By Definition \ref{def:meas}, $(\mathcal{P}_{i+1},\P_{i+1})$ satisfies the desired property.
\end{proof}

\begin{remark}
This lemma gives an Alberti representation of the measure $\mu_i$. In \cite{Bat15}, Bate used a property he called \emph{universality} of Alberti representations to characterize Lipschitz differentiability spaces. Our representation of the measure $\mu_\infty$ (which can be constructed by taking limits of the representations of $\mu_i$) will generally fail this universality condition, which is consistent with our discussion in Section \ref{ss:proofmethods} that $(X_\infty,d,\mu_\infty)$ is not a true Lipschitz differentiability space.
\end{remark}

\begin{proof}[Proof of Theorem \ref{thm:Sinfty}]
Let $i \geq 0$ and $E_{\text{circ}}(X_i')$ the set of edges $e' \in E(X_i')$ such that $(\pi^{i+1}_{i})^{-1}(e')$ is a circle. Set $S_i := (\pi^\infty_i)^{-1}(E_{\text{circ}}(X_i'))$, so $S_i$ is closed. By \eqref{eq:measpush}, \eqref{eq:measdisint}, and Axiom \ref{ax:thickcircsetlength},
$$\mu_\infty(S_i) \overset{\eqref{eq:measpush}}{=} \mu_i(E_{\text{circ}}(X_i')) \overset{\eqref{eq:measdisint}}{=} \sum_{P \in \mathcal{P}_i} \P_i(P) |E_{\text{circ}}(X_i') \cap P| \overset{\ref{ax:thickcircsetlength}}{\geq} \sum_{P \in \mathcal{P}_i} \P_i(P) \beta = \beta$$
Because of this, we set $S_\infty := \limsup_{i \to \infty} S_i$ (an $F_{\sigma\delta}$, and hence Borel, set) and get
\begin{equation*}
\mu_\infty(S_\infty) \geq \beta > 0
\end{equation*}
By definition, $S_\infty$ has the following property: for any $x \in S_\infty$, there is a subsequence $i_j(x)$ of $i$ for which $\pi_{i_j(x)}^\infty(x) \in E_{\text{circ}}(X_{i_j(x)}')$. Thus, each pointed metric space $(X_\infty,\frac{1}{r_{i_j}(x)}d,x)$ contains a circle whose height (see Axiom \ref{ax:thickcircht} for definition of height) is bounded below by $\alpha$, and the point $x$ belongs to this circle. Let $\U(x)$ be any nonprincipal ultrafilter on $\N$ containing $\{i_j(x)\}_{j=0}^\infty$, which exists by Zorn's lemma. Then the $\U(x)$-ultralimit of this sequence of pointed metric spaces must also contain such a circle (and the point $x$ will again belong to this circle), which obviously doesn't topologically embed into $\R$.
\end{proof}

\begin{remark} \label{rmk:tangentconegeodesics}
As described in the proof, each of the pointed spaces $(X_\infty,\frac{1}{r_{i_j}(x)}d,x)$ contain a circle of height $\alpha$ which contains $x$. Let $e$ and $e_\op$ be the opposite edges of this circle. We can extend $e$ in both directions to a 0-1 edge path. Since $e_\op$ has the same vertices as $e$, this also extends $e_\op$ to a 0-1 edge path. Unioning the circle $e \cup e_\op$ with the extension to a 0-1 edge path results in a space consisting of two 0-1 geodesics whose union contains a circle of height $\alpha$, and that coincide with each other outside that circle. Passing to the ultralimit, we see that the tangent cone $T_x^{r_i(x),\U(x)}X_\infty$ contains two bi-infinite geodesics whose union contains a circle of height $\alpha$, and that coincide with each other outside that circle. Both geodesics get mapped down isometrically onto $\R$ under the blowup $(\pi^\infty_0)_x: T_x^{r_i(x),\U(x)}X_\infty \to \R$.
\end{remark}


\section{Approximation of Functions on $X_\infty$ via $X_i$}
\label{sec:approx}
We begin this section by introducing our fundamental tool for approximating functions on $X_\infty$ by functions on $X_i$, the conditional expectation. The main results are Theorems \ref{thm:condexpbnd} and \ref{thm:condexpmeas}. We then use this tool to define the derivative of Lipschitz functions on $X_\infty$. The main result on the derivative is Theorem \ref{thm:derivdef}.


\subsection{Conditional Expectation} \label{ss:condexp}
Let $i \geq 0$ and $j \in \{i, i+1, \dots \infty\}$.

\begin{definition}
The \textbf{conditional expectation} is a bounded linear map $\E^{j}_{i}: L^1(\mu_{j};B) \to L^1(\mu_i;B)$ uniquely characterized by the identity
\begin{equation} \label{eq:condexp}
\int_{X_i} \phi \cdot \E^{j}_{i}(h) d\mu_i = \int_{X_j} (\phi \comp \pi^j_i) \cdot h d\mu_j
\end{equation}
for all $h \in L^1(\mu_{j};B)$ and $\phi \in L^\infty(\mu_i)$. It is a standard tool in probability theory whose existence can be proven by elementary theorems of measure theory. See Chapter 1 of \cite{Pi16} for background.
\end{definition}

It follows from $L^p$-$L^q$ duality that the conditional expectation is also contractive from $L^p(\mu_j;B) \to L^p(\mu_i;B)$ for any $p \in [1,\infty]$. The majority of this section is dedicated to proving the following theorem:

\begin{theorem} \label{thm:condexpbnd}
For every $i \geq 0$, $\E^\infty_i$ maps $\Lipo{X_\infty;B}$ into $\Lipo{X_i;B}$ with operator norm bounded by $L^2$.
\end{theorem}

Such a result does not hold for general metric measure spaces (easy examples on $[0,1]$ show that conditional expectation need not preserve Lipschitz or even continuous functions), but will in our specific instance.

The proof will come at the end of this subsection and is preceded by  several lemmas. We give an outline of the proof structure here:

\begin{itemize}
\item Show that for every $j < \infty$, $\E^{j}_{i}: \Lipo{X_j;B} \to \Lipo{X_i;B}$ has operator norm uniformly bounded by $L$.
\item Noting that $\E^{j}_{i} := \E^{{i+1}}_{i} \comp \E^{{i+2}}_{{i+1}} \comp \dots \E^{{j}}_{{j-1}}$, to prove the previous item, it suffices to consider the case $j = i+1$ and prove that $\|\E^{i+1}_i\|_{\Lipo{X_{i+1;B}} \to \Lipo{X_{i};B}} \leq 1 + \delta_i'$, because by \eqref{eq:pi-1Lipbndprod} we obtain 
\begin{equation} \label{eq:finitecondexpbnd}
\|\E^{j}_i\|_{\Lipo{X_{j}} \to \Lipo{X_{i}}} \leq \Pi_{k=i}^{j-1}(1 + \delta_k') \leq L
\end{equation}
for every $\infty > j \geq i \geq 0$. This is accomplished with Lemma \ref{lem:finitecondexpbnd}.
\item Extend the domain to $X_\infty$ by approximating with maps factoring through some $X_i$, Lemma \ref{lem:Lunifdense} (we gain another factor of $L$ here).
\end{itemize}

\subsubsection{Explicit Formula for and Boundedness of $\E^{i+1}_i$}
\begin{lemma} \label{lem:finitecondexpbnd}
For each $i \geq 0$ and $h \in \Lipo{X_{i+1};B}$,
\begin{equation} \label{eq:condexpform}
[\E^{i+1}_i(h)](p) = \frac{h(p)+h(p^\op)}{2}
\end{equation}
(recall the definition of $p^\op$ from Axiom \ref{ax:graph2}\ref{subax:graph2opintcirc}). Furthermore,
$$\|\E^{i+1}_i\|_{\Lipo{X_{i+1}} \to \Lipo{X_{i}}} \leq 1 + \delta_i'$$
\end{lemma}

\begin{proof}
Let $i \geq 0$ and $h \in \Lipo{X_{i+1}}$. It is a relatively simple exercise to check that \eqref{eq:condexpform} satisfies \eqref{eq:condexp} using Definition \ref{def:meas}. We now bound the operator norm. Let $x,y \in X_i$. No two points of $X_i \sbs X_{i+1}$ can belongs to opposite edges of a circle in $X_{i+1}$, so also $x^\op$ and $y^\op$ do not belong to opposite edges of a circle. Thus the hypotheses for Lemma \ref{lem:pi-1Lipbnd} are met. Then
$$\|\E^{i+1}_i(h)(x) - \E^{i+1}_i(h)(y)\| = \frac{\|h(x)+h(x^\op)-h(y)-h(y^\op)\|}{2}$$
$$\leq \frac{\|h(x)-h(y)\|}{2} + \frac{\|h(x^\op)-h(y^\op)\|}{2} \leq \frac{\|h\|_{\Lipo{X_{i+1}}}}{2}(d(x,y) + d(x^\op,y^\op))$$
$$\overset{\text{Lemma }\ref{lem:pi-1Lipbnd}}{\leq} \frac{\|h\|_{\Lipo{X_{i+1}}}}{2}(d(x,y) + (1+2\delta_i')d(x,y)) = (1+\delta_i')\|h\|_{\Lipo{X_{i+1}}}d(x,y)$$
\end{proof}

\subsubsection{Extending Domain to $\Lipo{X_\infty;B}$}
For $Y$ a metric space and $K \geq 1$, we say a subspace $V \sbs \Lipo{Y;B}$ is $K$-\emph{uniformly dense} in $\Lipo{Y;B}$ if the closure with respect to the topology of uniform convergence of compacta (equivalently, pointwise convergence on any dense subset) of the ball of radius $K$ in $V$ contains the unit ball of $\Lipo{Y;B}$. 

Each Banach space $\Lipo{X_i;B}$ can be identified as a closed subspace of $\Lipo{X_\infty;B}$ by pulling back under the map $\pi^\infty_i$. Denote the image of this identification by $\Lipo{X_i;B}_\pi$. We then obtain the (nonclosed) subspace $\cup_{i < \infty} \Lipo{X_i;B}_\pi \sbs \Lipo{X_\infty;B}$. We note that, for any $f \in \Lipo{X_i;B}$,
$$\|f\|_{\Lipo{X_i;B}} \leq \|f \comp \pi^\infty_i\|_{\Lipo{X_\infty;B}} \leq \|f\|_{\Lipo{X_i;B}}\|\pi^\infty_i\|_{\text{Lip}} \leq L\|f\|_{\Lipo{X_i;B}}$$
so that the embeddings $\Lipo{X_i;B}_\pi \hookrightarrow \Lipo{X_\infty;B}$ are uniformly bounded but not isometric.

\begin{lemma} \label{lem:Lunifdense}
For any Banach space $B$, $\cup_{i < \infty} \Lipo{X_i;B}_\pi \sbs \Lipo{X_\infty;B}$ is $L$-uniformly dense.
\end{lemma}

\begin{proof}
Let $f$ be in the unit ball of $\Lipo{X_\infty;B}$. Let $g_i$ be the restriction to $X_i$ of $f$. Then $g_i$ belongs to the unit ball of $\Lipo{X_i;B}$. Then $g_i \comp \pi^\infty_i$ belongs to the ball of radius $L$ of $\cup_{i < \infty} \Lipo{X_i;B}_\pi$. Clearly $g_i \comp \pi^\infty_i$ converges pointwise to $f$ on the dense subset $X_{< \infty}$.
\end{proof}

\begin{proof}[Proof of Theorem \ref{thm:condexpbnd}]
Let $i \geq 0$. Let $f$ be in the unit ball of $\Lipo{X_\infty;B}$. Let $f_j$ be a sequence in the ball of radius $L$ of $\cup_{i < \infty} \Lipo{X_i;B}_\pi$ converging uniformly to $f$, which exists by Lemma \ref{lem:Lunifdense}. Then since $\E^\infty_i$ is bounded on $L^\infty$, $\E^\infty_i(f_j)$ converges uniformly to $\E^\infty_i(f)$. Furthermore, for every $j$, by Lemma \ref{lem:finitecondexpbnd} and \eqref{eq:finitecondexpbnd}, $\|\E^\infty_i(f_j)\|_{\Lipo{X_i;B}} \leq L\|f_j\|_{\Lipo{X_\infty;B}} \leq L^2$. This implies $\|\E^\infty_i(f)\|_{\Lipo{X_i;B}} \leq L^2$.
\end{proof}

\subsubsection{Measure Representation of Conditional Expectation}
We conclude our discussion of conditional expectation with a small theorem we will use once in the proof of Theorem \ref{thm:weakRNPdiff}. We begin with a standard but useful martingale convergence lemma.

\begin{lemma} \label{lem:condexpconv}
For any Lipschitz map $h: X_\infty \to \R$ (not necessarily vanishing at 0) and $i \geq 0$, $\E^\infty_i(h)$ is Lipschitz and $\E^\infty_i(h) \overset{i \to \infty}{\to} h$ uniformly.
\end{lemma}

\begin{proof}
Let $h: X_\infty \to \R$ be Lipschitz so that $h - h(0) \in \Lipo{X_\infty}$. Then Theorem \ref{thm:condexpbnd} implies $h = \E^\infty_i(h - h(0)) + h(0)$ is Lipschitz.

The Stone-Weierstrass theorem for algebras of continuous functions implies $\cup_{j < \infty} C(X_j)_\pi$ is uniformly dense in $C(X_\infty)$, where $C(X_j)_\pi$ is defined to be the continuous real-valued functions on $X_\infty$ factoring through $X_j$. Then since $\E^\infty_i(h) \overset{i \to \infty}{\to} h$ (since it is eventually constant) for all $h \in \cup_{j < \infty} C(X_j)$, since $\sup_i \|E^\infty_i\|_{L^\infty(\mu_\infty) \to L^\infty(\mu_\infty)} = 1 < \infty$, and since $\mu_\infty$ and $\mu_i$ are fully supported on $X_\infty$ and $X_i$, the claim follows.
\end{proof}

\begin{theorem} \label{thm:condexpmeas}
For each $i \geq 0$, and $p \in X_i$, there exists a unique Borel probability measure $\mu_\infty^p$ supported on $(\pi_i^\infty)^{-1}(p)$ such that for any $h \in C(X_\infty;B)$,
\begin{equation} \label{eq:condexpmeas}
[\E^\infty_{i}(h)](p) = \int_{(\pi_i^\infty)^{-1}(p)} h d\mu_\infty^p
\end{equation}
\end{theorem}

\begin{proof}
Let $p \in X_i$. First we assume $B = \R$. Since, by Lemma \ref{lem:condexpconv} and the usual Stone Weierstrass theorem, $\E^\infty_i$ preserves continuous functions and has uniform-uniform operator norm 1 (since $\mu_\infty$ and $\mu_i$ are fully supported), the map $h \mapsto [\E^\infty_i(h)](p)$ is a norm 1 linear functional on $C(X_\infty)$. Further, if $h \geq 0$, $[\E^\infty_i(h)](p) \geq 0$. Thus, our linear functional is represented by a probability measure $\mu_\infty^p$ on $X_\infty$. It remains to show $\mu_\infty^p$ is supported on $(\pi_i^\infty)^{-1}(p)$. Consider the Lipschitz function $h_p: X_\infty \to R$ defined by $h_p(x) = d(x,(\pi_i^\infty)^{-1}(p))$. This function vanishes on $(\pi_i^\infty)^{-1}(p)$ and is strictly positive on $X_\infty \setminus (\pi_i^\infty)^{-1}(p)$. Thus, it suffices to show $[\E^\infty_i(h_p)](p) = 0$. Let $\eps > 0$. By Lemma \ref{lem:condexpconv}, $\E^\infty_i(h_p) \overset{i \to \infty}{\to} h_p$ uniformly, so there exists $j \geq i$ such that $|[\E^\infty_j(h_p)](x)| < \eps$ for all $x \in (\pi^j_i)^{-1}(p)$ (since $h_p$ vanishes on $(\pi^\infty_i)^{-1}(p))$). Since $\E^\infty_j(h_p)$ is a Lipchitz function on $X_j$, we may apply \eqref{eq:condexpform} (this was originally stated for functions vanishing at 0 but easily extends to the general case) and induction to conclude $|[\E^j_i(\E^\infty_j(h_p))](p)]| < \eps$. Since $\E^\infty_j \comp \E^j_i = \E^\infty_i$, we take $\eps \to 0$ and obtain the desired conclusion for $B = \R$. 

Now we extend to general $B$. Define a map $E: C(X_\infty;B) \to C(X_i;B_{\text{weak}})$ by
$$[E(h)](p) := \int_{(\pi_i^\infty)^{-1}(p)} h d\mu_\infty^p$$
where $B_{\text{weak}}$ indicated the space $B$ equipped with the weak topology. We need to show $E = \E^\infty_i$, which we already know holds for $B = \R$. First, let us quickly verify that $E$ indeed maps into the desired space. Let $h \in C(X_\infty;B)$ and $b^* \in B^*$. By an elementary property of the Bochner integral (see Chapter 1 of \cite{Pi16}, especially (1.7)) and the fact that $E = \E^\infty_i$ on real-valued continuous functions, $b^* \comp E(h) = E(b^* \comp h) = \E^\infty_i(b^* \comp h)$. We already know $\E^\infty_i$ maps real-valued continuous functions to real-valued continuous functions, so this shows $b^* \comp E(h)$ is continuous, completing our verification. By another elementary fact on $B$-valued conditional expectation (again see see Chapter 1 of \cite{Pi16}, (1.7)), $\E^\infty_i(b^* \comp h) = b^* \comp \E^\infty_i(h)$ $\mu_i$-almost everywhere, for every $b^* \in B^*$. Thus, $\mu_i$-almost everywhere, $b^* \comp E(h) = b^* \comp \E^\infty_i(h)$ for every $b^* \in B$, implying $E(h) = \E^\infty_i(h)$ $\mu_i$-almost everywhere. But since both $E(h)$ and $\E^\infty_i(h)$ are continuous functions from $X_i$ into the Hausdorff space $B_{\text{weak}}$, and since $\mu_i$ is fully-supported, $E(h) = \E^\infty_i(h)$ everywhere.
\end{proof}


\subsection{The Derivative and Fundamental Theorem of Calculus}
We define the derivative of Lipschitz functions on $X_\infty$ in this section. To do so, we must (and do) assume that $B$ has the RNP. We also prove an inequality in Theorem \ref{thm:ftc} that should be thought of as an adapted version of the fundamental theorem of calculus.

\begin{definition}
For any $h_i \in \cup_{j < \infty} \Lipo{X_j;B}$, since $X_i$ is a finite graph equipped with a measure mutually absolutely continuous with length measure and with a distance geodesic on edges, the fact that $B$ has the RNP allows us to take the derivative of $h_i$ $\mu_i$-almost everywhere defined by the usual formula $\boldsymbol{h_i'(x)} = \lim_{t \to 0} \frac{h_i(x+t)-h_i(x)}{t}$. We make sense of $x+t$ for $t$ small by identifying the directed edge contained $x$ with an interval, and the limit is an almost everywhere, norm limit. Equivalently, $h_i'$ is characterized by
\begin{equation} \label{eq:derivdef}
\lim_{r \to 0} \sup_{y \in B^i_r(x)} \frac{\|h_i(y)-h_i(x)-h_i'(x)(\pi(y)-\pi(x))\|}{r} = 0
\end{equation}
for $\mu_i$-almost every $x \in X_i$, where $\pi := \pi^\infty_0$. The map $h_i \mapsto h_i'$ is a linear contraction $\Lipo{X_\infty;B} \to L^\infty(\mu_\infty;B)$
\end{definition}

\begin{theorem} \label{thm:derivdef}
There exists a unique bounded linear map $h \mapsto h': \Lipo{X_\infty;B} \to L^\infty(\mu_\infty;B)$, called the \textbf{derivative}, that
\begin{enumerate}
\item satisfies $\E^\infty_i(h)' \overset{i \to \infty}{\to} h'$ $\mu_\infty$-almost everywhere
\item restricts to the usual derivative on $\cup_{j < \infty} \Lipo{X_j;B}$
\item has operator norm bounded by $L^2$.
\end{enumerate}
\end{theorem}

\begin{proof}
Note that uniqueness and the second statement already follow from the first statement. Let $h \in$ Lip$_0(X_\infty;B)$ with $\|h\|_{\Lipo{X_\infty;B}} \leq 1$, and for any $i \geq 0$, let $h_i := \mathbb{E}^{\infty}_{i}(h)$, so that $\|h_i\|_{\Lipo{X_i;B}} \leq L^2$ (by Theorem \ref{thm:condexpbnd}). Then the intermediate averages $x \mapsto \frac{h_i(x+t)-h_i(x)}{t}$ are uniformly (in $t$) $L^\infty(\mu_i;B)$-bounded by $L^2$. The DCT then implies that $\frac{h_i(\cdot+t)-h_i(\cdot)}{t} \overset{t \to 0}{\to} h_i'(\cdot)$ in $L^1(\mu_i;B)$. Then since the conditional expectation $\mathbb{E}^{{i+1}}_{i}: L^1(\mu_{i+1};B) \to L^1(\mu_{i};B)$ is continuous,
$$\mathbb{E}^{i+1}_{i}(h_{i+1}') = \mathbb{E}^{i+1}_{i}\left(\lim_{t \to 0}\frac{h_{i+1}(\cdot+t)-h_{i+1}}{t}\right) = \lim_{t \to 0}\mathbb{E}^{i+1}_{i}\left(\frac{h_{i+1}(\cdot+t)-h_{i+1}}{t}\right)$$
$$= \lim_{t \to 0}\frac{\left[\mathbb{E}^{i+i}_{i}(h_{i+1})\right](\cdot+t)-\left[\mathbb{E}^{i+i}_{i}(h_{i+1})\right]}{t} = \lim_{t \to 0} \frac{h_i(\cdot + t)-h_i}{t} = h_i'$$
The second to last equality says that conditional expectation commutes with precomposition with a translation, which can be directly verified by \eqref{eq:condexpform}. Thus, the sequence $(h_i')_{i=0}^\infty$ forms a martingale uniformly bounded in $L^\infty(\mu_\infty;B)$ by $L^2$. Since $B$ has the RNP property, the martingale converges $\mu_\infty$-almost everywhere to some function in $L^\infty(\mu_\infty;B)$ with norm bounded by $L^2$. We define $h'$ to be this limit.
\end{proof}

\begin{theorem}[Fundamental Theorem of Calculus] \label{thm:ftc}
For all $g \in \Lipo{X_\infty;B}$, $i \geq 1$, $e \in E(X_{i-1})$, and $x,y \in (\pi_{i-1}^i)^{-1}(e)$,
\begin{equation*}
\left\|[\E^\infty_i(g)](y)-[\E^\infty_i(g)](x)\right\| \leq 2|y-x|\fint_{[x,y]} \E^\infty_i(\|g'\|) d\mu_i
\end{equation*}
\end{theorem}

\begin{proof}
Let $g$, $i$, $e$ and $x,y$ be as above. Set $g_i := \E^\infty_i(g)$. First assume that $x$ and $y$ belong to a directed edge path. Then the usual Lebesgue fundamental theorem of calculus implies $\int_x^y g_i' ds = g_i(y)-g_i(x)$, where, for any positive Radon $\nu$ on $X_i$ and $f \in L^1(X_i,\nu;B)$, $\int_x^y f d\nu$ is interpreted as $\int_{[x,y]} f d\nu$ if $x \leq y$ along the path, and $-\int_{[y,x]} f d\nu$ if $y \leq x$. If $x$ and $y$ don't belong to a directed edge path, there exists an intermediate point $z$ on the shortest path from $x$ to $y$ such that the path is directed from $x$ to $z$, and then anti-directed from $z$ to $y$, or vice versa. We then still have $\int_x^y g_i' ds = g_i(y)-g_i(x)$ if we interpret $\int_x^y f d\nu$ as $\int_{[x,z]} f d\nu - \int_{[y,z]} f d\nu$ if $x \leq z$ and $y \leq z$ or $-\int_{[z,x]} f d\nu + \int_{[z,y]} f d\nu$ if $z \leq x$ and $z \leq y$. For future use, we also note that $\left\|\int_x^y f d\nu\right\| \leq \int_{[x,y]} \|f\| d\nu$.

As explained in the proof of Lemma \ref{lem:doubling}, $\mu_i$ restricted to $[x,y] \sbs (\pi_{i-1}^i)^{-1}(e)$ is bounded below by $\frac{c}{2}$ times length measure and above by $c$ times length measure. This implies that for any $f \in L^1(\mu_i)$ with $f \geq 0$, we have
$$\frac{1}{2}\fint_{[x,y]} f ds \leq |y-x|\fint_{[x,y]} f d\mu_i \leq 2\fint_{[x,y]} f ds$$
Combining the last two paragraphs yields:
$$\|g_i(y)-g_i(x)\| = \left\|\int_x^y g_i' ds\right\| \leq \int_{[x,y]} \|g_i'\| ds \leq 2|y-x|\fint_{[x,y]} \|g_i'\| d\mu_i$$
$$ = 2|y-x|\fint_{[x,y]} \|\E^\infty_i(g)'\| d\mu_i = 2|y-x|\fint_{[x,y]} \|\E^\infty_i(g')\| d\mu_i$$
\end{proof}


\section{Maximal Operator and $L^1 \to L^{1,w}$ Inequality}
\label{sec:maximalineq}
\begin{definition}
Let $i \geq 0$ and $h_i \in L^1(\mu_i)$. For any nonvertex $x_i \in X_i$ and $i \geq 0$, define
\begin{equation*}
[M_i(h_i)](x_i) := \left(\sup_{y_i \in (\pi^i_{i-1})^{-1}(e_{i-1}(x_i))} \fint_{[x_i,y_i]} |h_i|d\mu_i\right)
\end{equation*}

Now let $h \in L^1(\mu_\infty)$, and set $h_i := \E^\infty_i(h)$. For any nonvertex $x \in X_\infty$, set $x_i := \pi_i^\infty(x)$ and define the \textbf{maximal function}

\begin{equation} \label{eq:maximaldef}
[\boldsymbol{M(h)}](x) := \sup_{i \geq 0} [M_i(h_i)](x_i)
\end{equation}
\end{definition}

\begin{theorem}[Maximal Inequality] \label{thm:maximalineq}
There exists a constant $C \geq 1$ such that for any $h \in L^1(\mu_\infty;B)$ and $p \in (1,\infty]$,
\begin{equation}
\|M(\|h\|)\|_{L^{1,w}(\mu_\infty)} \leq \frac{Cp}{p-1}\|h\|_{L^p(\mu_\infty;B)}
\end{equation}
\end{theorem}

\begin{proof}
As is typical, the proof is an application of a relevant covering lemma, Lemma \ref{lem:coveringlemma}, which we state and prove following this proof. This lemma is a combination of the Vitali covering lemma for doubling metric measure spaces and the covering lemma for atoms in a filtration of finite $\sigma$-algebras. Let $h \in L^1(\mu_\infty;B)$, $h_i := \E^\infty_i(h)$, and $p \in (1,\infty]$. After making the usual ``covering lemma-to-maximal inequality" argument, we will have a $C \geq 1$ (independent of $h$ or $p$, given to us by Lemma \ref{lem:coveringlemma}) such that
$$\|M(\|h\|)\|_{L^{1,w}(\mu_\infty)} \leq C\|h^*\|_{L^1(\mu_\infty)}$$
where $h^*$ is Doob's maximal function; $h^*(x) := \sup_{i \geq 0} \|h_i\|(x)$. By Doob's maximal inequality (see Theorem 1.25 of \cite{Pi16}),
$$\|h^*\|_{L^p(\mu_\infty)} \leq \frac{p}{p-1}\|h\|_{L^p(\mu_\infty;B)}$$
Combining these two inequalities with the simple inequality $\|h^*\|_{L^1(\mu_\infty)} \leq \|h^*\|_{L^p(\mu_\infty)}$ yields the desired conclusion.
\end{proof}

\begin{lemma}[Covering Lemma] \label{lem:coveringlemma}
Let $\Gamma$ be a collection of closed subsets of $X_\infty$, such that for each $\gamma \in \Gamma$, there is an $i \geq 1$, a (not necessarily directed) shortest path $[p_\gamma,q_\gamma] \sbs X_i$, and an edge $e_\gamma \in X_{i-1}$ such that:
\begin{itemize}
\item $\gamma = (\pi_i^{\infty})^{-1}([p_\gamma,q_\gamma])$
\item $[p_\gamma,q_\gamma]$ is completely contained in $(\pi_{i-1}^i)^{-1}(e_\gamma)$.
\end{itemize}
Then there exists a subfamily $\Gamma' \sbs \Gamma$, such that
\begin{itemize}
\item The sets in $\Gamma'$ are essentially pairwise disjoint
\item For each $\gamma' \in \Gamma'$, there exists a closed set containing $\gamma'$, denoted $\gamma'_C$, such that $\bigcup_{\gamma' \in \Gamma'} \gamma'_C \sps \bigcup \Gamma$ and $\mu_\infty(\gamma'_C) \leq C\mu_\infty(\gamma')$.
\end{itemize}
\end{lemma}

\begin{proof}
First, consider the collection of sets $E_\Gamma := \{(\pi_{i-1}^\infty)^{-1}(e_\gamma)\}_{\gamma \in \Gamma}$. This set covers $\bigcup \Gamma$ by assumption. It is a collection of atoms in the filtration $(\A_i)_{i=0}^\infty$, where $\A_i$ is the $\sigma$ algebra on $X_\infty$ generated by preimages of edges in $E(X_i)$ under the map $\pi_i^\infty$. Thus we may find an essentially disjoint subcollection that still covers $\cup \Gamma$. We consider a single one these sets, $(\pi_{i-1}^\infty)^{-1}(e)$. Let $\Gamma_e$ be the collection of those $\gamma \in \Gamma$ with $[p_\gamma,q_\gamma] \sbs (\pi_{i-1}^i)^{-1}(e)$. Since preimages under $\pi_i^\infty$ preserve unions and essential disjointness, it suffices to work directly with the paths $[p_\gamma,q_\gamma]$. $[p_\gamma,q_\gamma]$ is contained in a geodesic ball $B_r(p_\gamma)$, where $r = |p_\gamma-q_\gamma|$. By Lemma \ref{lem:doubling}, $\mu_i(B_r(p_\gamma)) \leq 4\mu_i([p_\gamma,q_\gamma])$. By the 5r covering lemma, we can then find a pairwise disjoint subcollection of $\{B_r(p_\gamma)\}_{\gamma \in \Gamma_e}$, say $\{B_r(p_{\gamma'})\}_{\gamma' \in \Gamma_e'}$, such that $\{B_{5r}(p_\gamma')\}_{\gamma' \in \Gamma_e'}$ covers $\bigcup\{B_r(p_\gamma)\}_{\gamma \in \Gamma_e}$ (and thus covers $\bigcup \Gamma_e$). We set $\Gamma' := \bigcup_{e \in E_\Gamma} \{(\pi_i^{\infty})^{-1}([p',q'])\}_{[p',q'] \in \Gamma_e}$. By Lemma \ref{lem:doubling}, $\mu_i(B_{5r}(p)) \leq \mu_i(B_{8r}(p)) \leq 4^3\mu_i(B_r(p))$. We set $\Gamma' = \bigcup_{e \in E} \Gamma_e'$, $C = 4\cdot4^3$, and $\gamma'_C = B_{5r}(p_{\gamma'})$.
\end{proof}


\section{Proof of Weak Form of RNP Differentiability, Theorem \ref{thm:weakRNPdiff}}
\label{sec:weakRNPdiff}
For each deep point $x \in D \sbs X_\infty$ (a full measure set), recall the natural scale $r_i(x) = |e_i(x)|$, where $e_i(x)$ is the unique edge of $X_i$ containing $\pi^\infty_i(x)$. Let $\pi := \pi^\infty_0$.

\begin{theorem} \label{thm:weakRNPdiff}
For every RNP space $B$ and Lipschitz map $f: X_\infty \to B$, for $\mu_\infty$-almost every $x \in X_\infty$, $f$ is differentiable at $x$ with respect to $\pi$ along the sequence of scales $(r_i(x))_{i=0}^\infty$. More specifically, for almost every $x \in D$ and any $R \geq 1$,
$$\limsup_{i \to \infty} \sup_{y \in B_{Rr_i(x)}(x)} \frac{\|f(y)-f(x)-f'(x)(\pi(y)-\pi(x))\|}{r_{i}(x)} = 0$$
where $f'$ is the derivative of $f$ from Theorem \ref{thm:derivdef}.
\end{theorem}

\begin{proof}
Let $B$ be an RNP space, $f: X_\infty \to B$ Lipschitz, and $R \geq 1$. The conclusion of the theorem is clearly invariant under postcomposition of $f$ with a translation, so we may assume $f \in \Lipo{X_\infty;B}$. For each $n \geq 0$, let $f_n := \E^{\infty}_{n}(f) \comp \pi^\infty_n \in \Lipo{X_\infty;B}$ (see Section \ref{ss:condexp} for relevant definitions). Let
$$(*) := \limsup_{i \to \infty} \sup_{y \in B_{Rr_i(x)}(x)} \frac{\|f(y)-f(x)-f'(x)(\pi(y)-\pi(x))\|}{r_{i}(x)}$$
(so $(*)$ is a function of $x$). For every $x$, the triangle inequality implies
$$(*) \leq \limsup_{n \to \infty} \limsup_{i \to \infty} \sup_{y \in B_{Rr_i(x)}(x)} \frac{\|f(y)-f(x)-f_n(y)+f_n(x)\|}{r_{i}(x)}$$
$$+ \frac{\|f_n(y)-f_n(x)-f_n'(x)(\pi(y)-\pi(x))\|}{r_{i}(x)}$$
$$+ \frac{\|(f_n'(x)-f'(x))(\pi(y)-\pi(x))\|}{r_{i}(x)}$$
For almost every $x$ and every fixed $n$, the second term equals 0 by \eqref{eq:derivdef}, and so
$$(*) \leq \limsup_{n \to \infty} \limsup_{i \to \infty} \sup_{y \in B_{Rr_i(x)}(x)} \frac{\|f(y)-f(x)-f_n(y)+f_n(x)\|}{r_{i}(x)}$$
$$+ \frac{\|(f_n'(x)-f'(x))(\pi(y)-\pi(x))\|}{r_{i}(x)}$$
$$\leq \limsup_{n \to \infty} \limsup_{i \to \infty} \sup_{y \in B_{Rr_i(x)}(x)} \frac{\|f(y)-f(x)-f_n(y)+f_n(x)\|}{r_{i}(x)} + LR\|f_n'(x)-f'(x)\|$$
By Theorem \ref{thm:derivdef}, the second term here also equals 0 for almost every $x$, and so
$$(*) \leq \limsup_{n \to \infty} \limsup_{i \to \infty} \sup_{y \in B_{Rr_i(x)}(x)} \frac{\|f(y)-f(x)-f_n(y)+f_n(x)\|}{r_{i}(x)}$$

Let $k_n := f-f_n$, so that $\sup_n \|k_n'\|_{L^\infty(\mu_\infty;B)} \leq \sup_n \|k_n\|_{\text{Lip}_0(X_\infty;B)} \leq 2L^2\|f\|_{\text{Lip}_0(X_\infty;B)}$ and $\|k_n'\| \overset{n \to \infty}{\to} 0$ $\mu_\infty$-almost everywhere (again by Theorem \ref{thm:derivdef}. This means $k_n'$ \emph{boundedly converges} to 0, and we will apply the DCT theorem later). It suffices to prove
$$\limsup_{n \to \infty} \limsup_{i \to \infty} \sup_{y \in B_{Rr_i(x)}(x)} \frac{\|k_n(y)-k_n(x)\|}{r_{i}(x)} = 0$$
Define $y_i := \pi^\infty_{i}(y)$ and $x_i := \pi^\infty_{i}(x)$. then by Theorem \ref{thm:condexpmeas},
$$[\E^\infty_i(k_n)](y_i) - [\E^\infty_i(k_n)](x_i) = \int_{(\pi^\infty_{i})^{-1}(y_i)} k_n d\mu_\infty^{y_i} - \int_{(\pi^\infty_{i})^{-1}(x_i)} k_n d\mu_\infty^{x_i}$$
Furthermore, for any $y \in (\pi^\infty_{i})^{-1}(y_i)$ and $x \in (\pi^\infty_{i})^{-1}(x_i)$, we have $\frac{d(y,y_i)}{r_i(x)},\frac{d(y,y_i)}{r_i(x)} \leq 2\delta'_{i}$ by \eqref{eq:fiberdiam}, which, together with the previous equation (and triangle inequality) gives us
$$\frac{\|k_n(y)-k_n(x)\|}{r_{i}(x)} \leq \frac{1}{r_i(x)}\int_{(\pi^\infty_{i})^{-1}(y_i)} \|k_n-k_n(y)\| d\mu_\infty^{y_i}$$
$$+ \frac{1}{r_i(x)}\left\|\int_{(\pi^\infty_{i})^{-1}(y_i)} k_n d\mu_\infty^{y_i} - \int_{(\pi^\infty_{i})^{-1}(x_i)} k_n d\mu_\infty^{x_i}\right\|$$
$$+ \frac{1}{r_i(x)}\int_{(\pi^\infty_{i})^{-1}(x_i)} \|k_n-k_n(x)\| d\mu_\infty^{x_i}$$
$$\leq \frac{\left\|[\E^\infty_i(k_n)](y_i) - [\E^\infty_i(k_n)](x_i)\right\|}{r_i(x)} + 4\|k_n\|_{\text{Lip}_0(X_\infty;B)}\delta'_{i}$$
$$\leq \frac{\left\|[\E^\infty_i(k_n)](y_i) - [\E^\infty_i(k_n)](x_i)\right\|}{r_i(x)} + 4(2L^2)\|f\|_{\text{Lip}_0(X_\infty;B)}\delta'_{i}$$
so it suffices to prove that
$$\limsup_{n \to \infty} \limsup_{i \to \infty} \sup_{y_i \in B^i_{2Rr_i(x)}(x_i)} \frac{\left\|[\E^\infty_i(k_n)](y_i) - [\E^\infty_i(k_n)](x_i)\right\|}{r_{i}(x)} = 0$$
for almost every $x$, where $B^i$ indicates a ball in the space $(X_i,d)$ (since $2Rr_i(x) \geq Rr_i(x) + \delta'_ir_i(x)$).

By Lemma \ref{lem:ballcontainedinedge}, for almost every $x$, if $i$ is sufficiently large (depending on $R$ and $x$), then $B^i_{2Rr_i(x)}(x_i)$ is completely contained in $(\pi^{i}_{i-1})^{-1}(e_{i-1}(x))$. Thus, by Theorem \ref{thm:ftc}, for such $i$ and any $y_i \in B^i_{2Rr_i(x)}(x_i)$,
$$\frac{\left\|[\E^\infty_i(k_n)](y_i) - [\E^\infty_i(k_n)](x_i)\right\|}{r_{i}(x)} \leq 2\frac{|y_i-x_i|}{r_i(x)}\fint_{[x_i,y_i]}\mathbb{E}^{\infty}_{i}(\|k_n'\|)d\mu_{i} =: (**)$$
By \eqref{eq:quasigeo}, since $y_i \in B^i_{2Rr_i(x)}(x_i) \sbs (\pi^{i}_{i-1})^{-1}(e_{i-1}(x))$, $|y_i-x_i| \leq 2Rr_i(x)$, and so
$$(**) \leq 4R\fint_{[x_i,y_i]}\mathbb{E}^{\infty}_{i}(\|k_n'\|)d\mu_{i} =: (***)$$
by the definition of $M$, the maximal operator defined by \eqref{eq:maximaldef}, we get 
$$(***) \leq 4R[M(\|k_n'\|)](x)$$
Then it suffices to show that
$$\limsup_{n \to \infty} [M(\|k_n'\|)](x) = 0$$
for almost every $x \in X_\infty$.

For this, it suffices to show that
$$\left\|\limsup_{n \to \infty} M(\|k_n'\|)\right\|_{L^{1,w}(\mu_\infty)} = 0$$
We have, by the DCT and Theorem \ref{thm:maximalineq},
$$\left\|\limsup_{n \to \infty} M(\|k_n'\|)\right\|_{L^{1,w}(\mu_\infty)} \overset{\text{DCT}}{=} \limsup_{n \to \infty} \left\|M(\|k_n'\|)\right\|_{L^{1,w}(\mu_\infty)}$$
$$\overset{\text{Theorem} \: \ref{thm:maximalineq}}{\leq} \limsup_{n \to \infty} 2C\left\|k_n'\right\|_{L^2(\mu_\infty;B)} \overset{\text{DCT}}{=} 0$$
\end{proof}


\section{Application to Non-BiLipschitz Embeddability}
\label{sec:application}
In this section we apply Theorem \ref{thm:mainthmssummary} to prove a new negative biLipschitz embeddability result, Corollary \ref{cor:Carnotapplication}.

\begin{proof}[Proof of Corollary \ref{cor:Carnotapplication}]
We'll proceed by contradiction. Let $G$ be a Carnot group, $B$ an RNP space, $M$ a metric space containing a thick family of geodesics, and $f = (f_1,f_2): M \to G \cross B$ a biLipschitz map. We may assume $M$ is complete. Then we let $X_\infty \sbs M$, $S_\infty \sbs X_\infty$, and $\mu_\infty$ be as in Theorem \ref{thm:mainthmssummary}, and from here on consider $f$ to be restricted to $X_\infty$.

Let $\psi: G: \R^k$ be the abelianization map. $\psi$ satisfies a well known unique lifting property: given any Lipschitz map $\gamma: \R \to \R^k$, there exists a unique Lipschitz lift $\tilde{\gamma}: \R \to G$, meaning $\psi \comp \tilde{\gamma} = \gamma$. Precomposing with $f$ gives a Lipschitz map $(\psi,id_B) \comp (f_1,f_2) = (\psi \comp f_1,f_2): X_\infty \to \R^k \oplus B$ into an RNP space.

By Theorem \ref{thm:mainthmssummary}, $(\psi \comp f_1,f_2)$ satisfies the weak form of differentiability $\mu_\infty$-almost everywhere. Pick a point $x \in S_\infty$ of differentiability (which exists since $\mu_\infty(S_\infty) > 0$) and an ultrafilter $\U(x)$ given to us by Theorem \ref{thm:mainthmssummary}. This means the blowup $(\psi \comp f_1,f_2)_x: T_x^{r_i(x),\U(x)}X_\infty \to \R^k \cross B$ exists and factors as $(\psi \comp f_1,f_2)_x = ((\psi \comp f_1)'(x),f_2'(x)) \comp \pi_x$, where $\pi_x: T_x^{r_i(x),\U(x)}X_\infty \to \R$ is the blowup of $\pi$ and $(\psi \comp f_1)'(x): \R \to \R^k$ and $f_2'(x): \R \to B$ are linear. Breaking these into components gives us the two factorizations
$$(\psi \comp f_1)_x = (\psi \comp f_1)'(x) \comp \pi_x$$
\begin{equation} \label{eq:blowupfactorization2}
(f_2)_x = f_2'(x) \comp \pi_x
\end{equation}

Let us consider the blowup map $(\psi \comp f_1)_x$. It turns out that both blowups $(f_1)_x: T_x^{r_i(x),\U(x)}X_\infty \to T_{f_1(x)}^{r_i(x),\U(x)}G = G$ and $\psi = \psi_{f_1(x)}: G = T_{f_1(x)}^{r_i(x),\U(x)}G \to \R^k$ exist and thus $(\psi \comp f_1)_x = \psi \comp (f_1)_x$. $(f_1)_x$ exists because the target space $G$ is proper, and $T_{f_1(x)}^{r_i(x),\U(x)}G = G$ because $G$ is proper and self-similar. Similar reasoning implies $\psi_{f_1(x)}: G \to \R^k$ exists and $\psi_{f_1(x)} = \psi$. Thus our first factorization can be re-expressed as
\begin{equation} \label{eq:blowupfactorization1}
\psi \comp (f_1)_x = (\psi \comp f_1)'(x) \comp \pi_x
\end{equation}
By Remark \ref{rmk:tangentconegeodesics}, there are two geodesics $\gamma,\gamma': \R \to T_x^{r_i(x),\U(x)}X_\infty$ whose combined image forms a circle of height $\alpha$, that coincide with each other outside that circle, and satisfy $\pi_x \comp \gamma = \pi_x \comp \gamma' = id_\R$. Using these equations, \eqref{eq:blowupfactorization1}, and \eqref{eq:blowupfactorization2} yields
$$\psi \comp (f_1)_x \comp \gamma = (\psi \comp f_1)'(x) = \psi \comp (f_1)_x \comp \gamma'$$
$$(f_2)_x \comp \gamma = f_2'(x) = (f_2)_x \comp \gamma'$$
Since $f_1,\gamma,\gamma'$ are Lipschitz, the unique lifting property of $\psi$ implies
$$(f_1)_x \comp \gamma = (f_1)_x \comp \gamma'$$
Combining these yields
$$(f_1,f_2)_x \comp \gamma = (f_1,f_2)_x \comp \gamma'$$

Since $(f_1,f_2)$ is biLipschitz, so is $(f_1,f_2)_x$. Thus, $\gamma = \gamma'$. This is a contradiction since the combined image of two equal geodesics would be a line and could not contain (even topologically) a circle.
\end{proof}


\section{Inverse Limit of Graphs in nonRNP Spaces}
\label{sec:invlimgraphsnonRNP}
In this section we modify the thick family of geodesics construction in \cite{Os14b} to obtain an embedding of an inverse limit of an admissible system of graphs into any nonRNP Banach space. To do so, we use the following characterization of nonRNP spaces (see Theorem 2.7 of \cite{Pi16}): for any nonRNP space $B$, there exist a $\delta > 0$ and an open, convex subset $C$ of the unit ball of $B$ such that for every $c \in C$, $c \in \text{co}(C \setminus B_{4\delta}(c))$.

\subsection{Generalized Diamond Systems}
\begin{definition}
\label{def:gendiamsys}
A \textbf{generalized diamond system} is an inverse system of connected metric graphs, $\dots  \overset{\pi^3_2}{\to} X_2 \overset{\pi^2_1}{\to} X_1 \overset{\pi^1_0}{\to} X_0$ satisfying:

\begin{enumerate}[label=(D\arabic*)]
\item \label{ax:gendiamsys1} $X_0$ has two vertices and one edge of length 1. We identify $X_0$ with $I := [0,1]$.
\item \label{ax:gendiamsys2} For any vertex $v \in V(X_i)$, $(\pi^{i+1}_i)^{-1}(\{v\})$ consists of a single vertex of $X_{i+1}$. We identify this vertex with $v$ and consider $V(X_i)$ as a subset of $V(X_{i+1})$.
\item \label{ax:gendiamsys3} There exist an $m_i$ and a subdivision $X_i'$ of $X_i$ so that:
	\begin{enumerate}[label=(\roman*)]
	\item For vertex $v \in V(X_i')$, $(\pi^{i+1}_i)^{-1}(\{v\})$ consists of one or two vertices of $X_{i+1}$. If $u,v$ are adjacent vertices in $X_i'$, then at most one of $(\pi^{i+1}_i)^{-1}(\{u\})$, $(\pi^{i+1}_i)^{-1}(\{v\})$ consists of two vertices.
	\item Each edge $e \in E(X_i)$ is subdivided into $2^{m_i}$ edges of $X_i'$ of equal length.
	\item $\pi^{i+1}_i: X_{i+1} \to X_i'$ is open, simplicial, and an isometry on every edge.
	\item  For any edge $e' \in E(X_i')$, $(\pi^{i+1}_i)^{-1}(e')$ consists of one or two edges, and if $e'$ is a terminal subedge of $e$ (meaning it shares a vertex with $e$), then $(\pi^{i+1}_i)^{-1}(e')$ consists of only one edge.
	\end{enumerate}
\end{enumerate}

A generalized diamond system admits a canonical sequence of Borel probability measures $(\mu_i)_{i=0}^\infty$ satisfying

\begin{enumerate}[resume*]
\item $\mu_0$ is Lebesgue measure on $I$.
\item Restricted to each edge of $X_i$, $\mu_i$ is a constant multiple of length measure.
\item For each $e' \in E(X_i')$, if $(\pi^{i+1}_i)^{-1}(e')$ consists of two edges, then the $\mu_{i+1}$ measure of each of these edges equals $\frac{1}{2}\mu_i(e')$, and if $(\pi^{i+1}_i)^{-1}(e')$ consists of one edge, then the $\mu_{i+1}$ measure of this edge equals $\mu_i(e')$.
\end{enumerate}

\begin{remark}
With a small adjustment, these axioms imply the axioms of an ``admissible" inverse system from \cite{CK15}. The only problem is that in \cite{CK15}, each edge of $X_i$ is subdivided into $m$ edges of $X_i'$, where $m$ is independent of $i$, and our subdivisions are into $2^{m_i}$ subedges, where $m_i$ can depend on $i$. To conform to the \cite{CK15} axiom, we can augment our inverse system by inserting extra graphs $X_i^j$ between $X_i$ and $X_{i+1}$ that are simply subdivisions of $X_i$ into $2^{j}$ subedges, for $1 \leq j \leq m_i$. The maps between them are identity maps. This new system will now be an admissible inverse system with subdivision parameter 2, and the inverse limit of the original system and augmented system will be the same. Thus, by Theorem 1.1 of \cite{CK15}, the inverse limit $(X_\infty,d_\infty,\mu_\infty)$ of a generalized diamond system is a PI space.
\end{remark}

There is one last axiom for a generalized diamond system which implies (10.3) from \cite{CK15} holds $\mu_\infty$-almost everywhere.

\begin{enumerate}[resume*]
\item \label{ax:gendiamsys7} For any edge $e \in E(X_i)$, every point in $(\pi_i^{i+1})^{-1}(e_{1/2})$ is at most 2 edge lengths (of $X_{i+1}$) away from a vertex of degree 4, where $e_{1/2}$ denotes the middle half of $e$.
\end{enumerate}

\end{definition}

\subsection{Proof of Theorem \ref{thm:gendiamembed}}
\begin{proof}[Proof of Theorem \ref{thm:gendiamembed}]
We begin by making some reductions. First, notice that it suffices to embed into $B \oplus_\infty \R$ for any nonRNP space $B$. This is because we may pick any closed, codimension-1 subspace $B' \sbs B$, which is also necessarily a nonRNP space, and get $B \cong B' \oplus_\infty \R$.

Let $B$ be a nonRNP space (in a slight abuse of notation, we'll use $\| \cdot \|$ to stand for both the norm on $B$ and the norm on $B \oplus_\infty \R$, but this shouldn't cause any confusion). We'll construct a sequence of subsets $(X_i)_{i=0}^
\infty$ of $B \oplus_\infty \R$ and maps $\pi^{i+1}_i: X_{i+1} \to X_i$ such that $(X_i,d_i)$ is a connected metric graph and $\dots  \overset{\pi^3_2}{\to} X_2 \overset{\pi^2_1}{\to} X_1 \overset{\pi^1_0}{\to} X_0$ is a generalized diamond system, where $d_i$ denotes the intrinsic metric on $X_i$ (shortest path metric, where path length is measured with respect to ambient Banach space). The construction will be such that there exist a $\delta > 0$ and $\delta_i > \delta$ such that $X_i$ is $\delta_i^{-1}$-quasiconvex in $B \oplus \R$, meaning $\delta_i d_i(x,y) \leq \|x-y\| \leq d_i(x,y)$ for all $x,y \in X_i$. Furthermore, the construction will be such that for any $v \in V(X_i) \sbs V(X_{i+1})$, $\pi^{i+1}_i(v) = v$ (see Axiom \ref{ax:gendiamsys2} for the identification of $V(X_i)$ as a subset of $V(X_{i+1})$). By density of the the vertices in the inverse limit space, this implies that the closure of $\cup_i V(X_i)$ in $B \oplus \R$ is $\delta^{-1}$-biLipschitz equivalent to the inverse limit of $\dots  \overset{\pi^3_2}{\to} X_2 \overset{\pi^2_1}{\to} X_1 \overset{\pi^1_0}{\to} X_0$.

Previously, we introduced geodesics as isometric maps on intervals, but in this proof it will be more convenient to consider the image of these maps instead of the map itself. For this reason, we use the term \emph{geodesic path} to mean the image of a geodesic map. Additionally, if $p$ and $q$ are points in a graph, we previously used the notation $|p-q|$ to denote the distance between $p$ and $q$ with respect to the length metric, but such notation will cause problems in this proof since we are working in a normed space. Instead, we will use the term \emph{intrinsic metric} which has the same meaning as length metric, and notation for this distance will be set subsequently.

\subsubsection{Model Graph}
\label{sss:modelgraph}
Let $\delta > 0$ and let $C$ be an open, convex subset of the unit ball of $B$ such that $0 \in C$ and $c \in \text{co}(C \setminus B_{4\delta}(c))$ for every $c \in C$, where $B_r(x)$ is the closed unit ball of radius $r$ centered at $x$. We describe how to construct a graph, for each $c \in C$, that will serve as a building block for the graphs $X_i$.

Let $c \in C$. We'll form two piecewise affine, geodesic paths from $(0,0)$ to $(c,1)$, denoted $\gamma_{0}(c)$ and $\gamma_{1}(c)$. The reader should refer to Figure \ref{fig:modelgraph} for a helpful visual of the construction. Since $c \in C$, $c = \alpha_1 c_1 + \dots \alpha_k c_k$ for some $\alpha_j \in (0,1)$ and $c_1, \dots c_k \in C$ with $\alpha_1 + \dots \alpha_k = 1$ and $\|c - c_j\| \geq 4\delta_c > 4\delta$ (note that since $c,c_j$ belong to the unit ball of $B$, $\delta_c \leq \frac{1}{2}$). Since $C$ is open, we may assume each $\alpha_j$ is a dyadic rational with common denominator $2^n$, by density of dyadic rationals in $[0,1]$. Additionally, by ``splitting" up terms of the form $\frac{m}{2^n}c_j$ into the $m$-fold sum $\frac{1}{2^n} c_j + \frac{1}{2^n} c_j + \dots \frac{1}{2^n} c_j$, we may assume $\alpha_j = 2^{-n_c}$ and $k = 2^{n_c}$ for some $n_c \geq 1$, independent of $j$ (of course we do not have that $\{c_j\}$ are distinct, but that is no issue). $\gamma_{0}(c)$ consists of a piecewise affine interpolation between $2 \cdot 2^{n_c} + 1$ vertices, $v_0, v_1', v_1, v_2', v_2, \dots v_{2^{n_c}}', v_{2^{n_c}}$. These vertices are such that $v_0 = (0,0)$, and for each $j$, $v_j' - v_{j-1} = 2^{-(n_c+1)}(c,1)$ and $v_j - v_j' = 2^{-(n_c+1)}(c_j,1)$. Likewise, $\gamma_{1}(c)$ consists of a piecewise affine interpolation between $2 \cdot 2^{n_c} + 1$ vertices, $w_0, w_1', w_1, w_2', w_2, \dots w_{2^{n_c}}', w_{2^{n_c}}$. These vertices are such that $w_0 = (0,0)$, and for each $j$, $w_j' - w_{j-1} = 2^{-(n_c+1)}(c_j,1)$ and $w_j - w_j' = 2^{-(n_c+1)}(c,1)$ (notice the flipping of primed and unprimed terms). It follows that $v_j = w_j$ for each $j$, and that $v_{2^{n_c}} = (c,1) = w_{2^{n_c}}$. These are indeed geodesic paths because the vectors $c, c_j$ all have norm 1 in $B$, and we take an $\infty$-norm direct sum. An isometry from these geodesics paths onto the interval $[0,1]$ is provided by projection onto the second coordinate.

$\gamma_0(c)$ is equipped with a graph structure. The vertex set is the ordered set $(v_0, v_1', v_1, v_2', v_2, \dots v_{2^{n_c}}', v_{2^{n_c}})$ and there is one edge between consecutive vertices consisting of the line segment between them. $\gamma_1(c)$ is similarly equipped with a graph structure. We let $\Gamma(c) = \gamma_0(c) \cup \gamma_1(c)$. Since $\gamma_0(c)$ and $\gamma_1(c)$ intersect only on their vertices, $\Gamma(c)$ inherits an induced graph structure. The vertex set is $\{v_0=w_0=(0,0),v_1',w_1',v_1=w_1, \dots v_{2^{n_c}}',w_{2^{n_c}}',v_{2^{n_c}}=w_{2^{n_c}}=(c,1)\}$. See Figure \ref{fig:modelgraph} for an example of $\Gamma(c)$ for $2^{n_c} = 4$. Loosely, $\Gamma(c)$ is made up of a sequence of parallelograms increasing in the ``$\R$ direction" of $B \oplus \R$ such that adjacent parallelograms share a common vertex. Because of this, for any two points of $x,y \in \Gamma(c)$ belonging to distinct parallelograms, the extrinsic distance $\|x-y\|$ and intrinsic distance $d_{in}(x,y)$ agree. We claim that each of these parallelograms is $\delta_c^{-1}$-quasiconvex. Then this claim together with the preceding sentence imply that $\Gamma(c)$ is $\delta_c^{-1}$-quasiconvex.

\emph{Proof of Claim.} Consider one of the parallelograms of $\Gamma(c)$. It has vertices $v_{j-1}=w_{j-1},v_j',w_j',v_j=w_j$ for some $j$. First notice that translations and dilations don't change the quasiconvexity constant of parallelograms, so we may perform such modifications to ours to obtain one that is easier to calculate with. Translate the parallelogram by $-v_{j-1}$ ($= -w_{j-1}$) so that one of the vertices is $(0,0)$, and the other vertices are $2^{-(n_c+1)}(c,1)$, $2^{-(n_c+1)}(c_j,1)$, and $2^{-(n_c+1)}(c+c_j,2)$. Then scale by $2^{n_c+1}$ so that the vertices are $(0,0)$, $(c,1)$, $(c_j,1)$, and $(c+c_j,2)$. Now we label the edges: let $e_1$ be the edge between $(0,0)$ and $(c,1)$, $e_2$ the edge between $(0,0)$ and $(c_j,1)$, $e_3$ the edge between $(c,1)$ and $(c+c_j,2)$, and $e_4$ be the edge between $(c_j,1)$ and $(c+c_j,2)$. Figure \ref{fig:parallelogram} shows an example of this parallelogram, and it will be helpful to keep this picture in mind while reading the remaining proof of the claim.

\begin{figure}
\includegraphics[scale=.5]{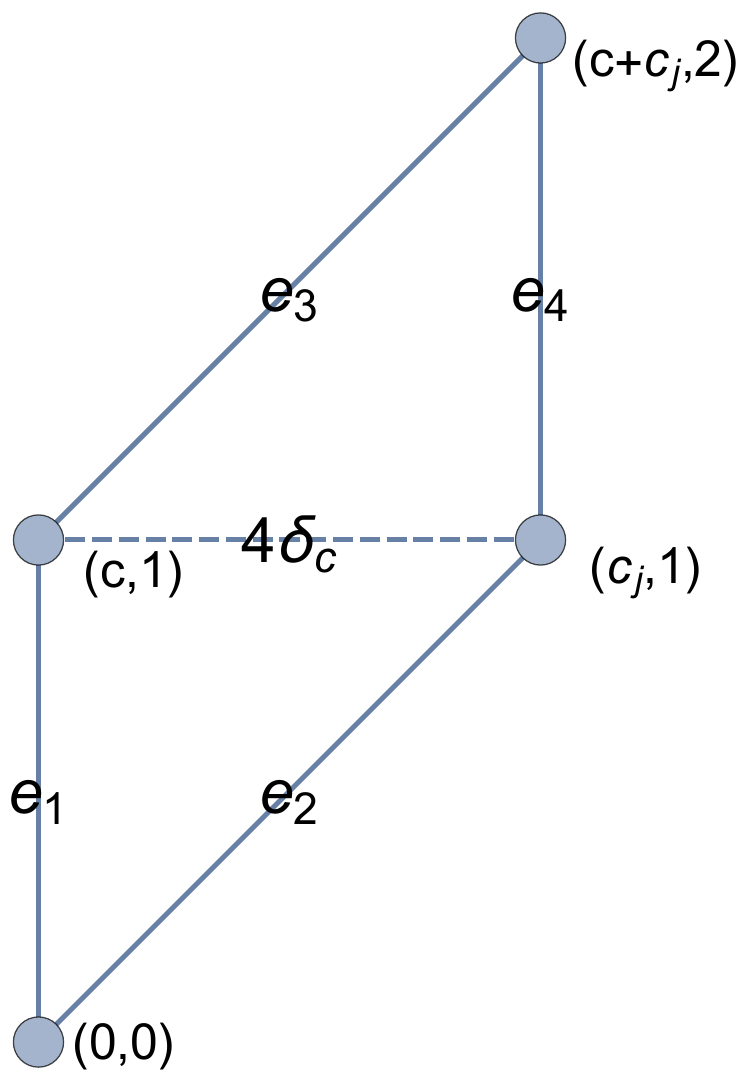}
\caption{The parallelogram with vertices $(0,0)$, $(c,1)$, $(c_j,1)$, and $(c+c_j,2)$. The horizontal axis is in the ``$B$ direction" of $B \oplus \R$, and the vertical axis is in the ``$\R$ direction". The extrinsic and intrinsic distance between any two points on $e_1 \cup e_3$ or any two points on $e_2 \cup c_4$ agree. The extrinsic distance between the two vertices $(c,1)$, $(c_j,1)$ is $4\delta_c$. All edge lengths are 1.}
\label{fig:parallelogram}
\end{figure}

Note that $e_1 \cup e_3$ is a subpath of the geodesic path corresponding to $\gamma_0(c)$, and $e_2 \cup e_4$ is a subpath of the geodesic path corresponding to $\gamma_1(c)$, so the intrinsic and extrinsic distance agree on these subsets. Let $x$ and $y$ be elements of the parallelogram. As just mentioned, if $x$ and $y$ belong to $e_1 \cup e_3$, or both belong to $e_2 \cup e_3$, then the intrinsic and extrinsic distance between $x$ and $y$ agree. Suppose then that $x$ belongs to $e_1$ and $y$ belongs to $e_2$. Then $x = \alpha(c,1)$ for some $\alpha \in [0,1]$, $y = \beta(c_j,1)$ for some $\beta \in [0,1]$, and the intrinsic distance between $x$ and $y$ is $\alpha + \beta$. Without loss of generality, assume $\beta \geq \alpha$, so that the intrinsic distance between $x$ and $y$, $d_{in}(x,y)$ is bounded by $2\beta$. Then the extrinsic distance between $x$ and $y$ is
$$\|x-y\| = \|\alpha(c,1)-\beta(c_j,1)\| = \|(\beta(c-c_j) + (\alpha-\beta)c,\alpha-\beta)\|$$
$$= \max(\|\beta(c-c_j) + (\alpha-\beta)c\|, |\alpha-\beta|)$$
$$\geq \max(\|\beta(c-c_j)\| - \|(\alpha-\beta)c\|, |\alpha-\beta|)$$
$$\geq \max(\|\beta(c-c_j)\| - |\alpha-\beta|, |\alpha-\beta|)$$
$$ \geq \max(\beta4\delta_c - |\alpha-\beta|, |\alpha-\beta|)$$
$$\geq (2\beta)\delta_c \geq \delta_c d_{in}(x,y)$$

\noindent showing that the quasiconvexity constant is bounded above by $\delta_c^{-1}$ in this case. By symmetry, we get the same upper bound if $x$ belongs to $e_3$ and $y$ belongs to $e_4$. There is one remaining case (since the rest of the cases follow from this one by symmetry), in which $x$ belongs to $e_1$ and $y$ belongs to $e_4$. In this case, $x = \alpha(c,1)$ for some $\alpha \in [0,1]$, $y = (c_j,1) + \beta (c,1)$ for some $\beta \in [0,1]$, and we use the trivial bound $d_{in}(x,y) \leq 2$ for the intrinsic distance. Then for the extrinsic distance, we have
$$\|x-y\| = \left\|\alpha(c,1) - \left((c_j,1) + \beta (c,1)\right)\right\|$$
$$= \left\|\left(\left(\alpha-\beta-1\right)c+(c-c_j),\alpha-\beta-1\right)\right\|$$
$$= \max\left(\left\|\left(\alpha-\beta-1\right)c+(c-c_j)\right\|,\left|\alpha-\beta-1\right|\right)$$
$$\geq \max\left(\left\|c-c_j\right\| - \left|\alpha-\beta-1\right|,\left|\alpha-\beta-1\right|\right)$$
$$\geq \max\left(4\delta_c - \left|\alpha-\beta-1\right|,\left|\alpha-\beta-1\right|\right)$$
$$\geq 2\delta_c \geq \delta_cd_{in}(x,y)$$
This completes the proof of the $\delta_c^{-1}$-quasiconvexity of the parallelogram.

\noindent \emph{End Proof of Claim}.

\begin{figure}
\includegraphics[scale=1]{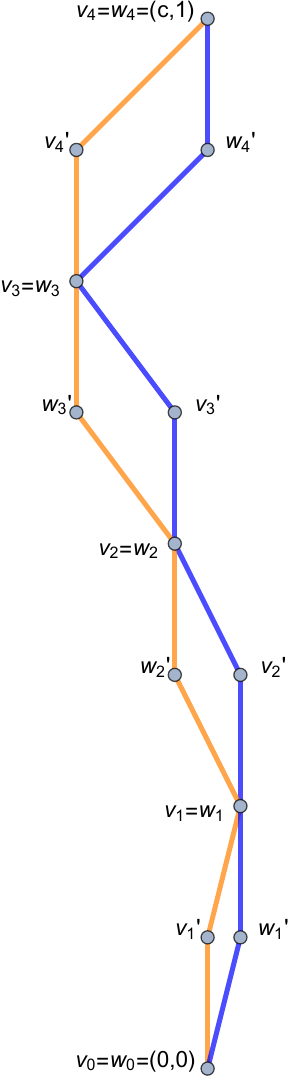}
\caption{The model graph $\Gamma(c)$ The geodesic path $\gamma_0(c)$ is shown in orange, and the geodesic path $\gamma_1(c)$ is shown in blue. The horizontal axis is in the ``$B$ direction" of $B \oplus \R$, and the vertical axis is in the ``$\R$ direction".}
\label{fig:modelgraph}
\end{figure}

Since $\gamma_0(c)$, $\gamma_1(c)$, and $[(0,0),(c,1)]$ are all geodesics with endpoints $(0,0)$ and $(c,1)$, there are unique isometries $\gamma_0(c) \to [(0,0),(c,1)]$ and $\gamma_1(c) \to [(0,0),(c,1)]$ fixing the endpoints. If we let $[(0,0),(c,1)]'$ denote the subdivision of $[(0,0),(c,1)]$ into subedges of length $2^{-(n_c+1)}$, the maps are graph isomorphisms. Combining these gives us a map $\pi_c: \Gamma(c) \to [(0,0),(c,1)]'$ which is open, simplicial, and an isometry on every edge. Furthermore, the preimage of any edge in $[(0,0),(c,1)]'$ consists of two edges of $\Gamma(c)$. Let $((0,0) = t_0, t_1', t_1, t_2', t_2, \dots t_{2^{n_c}}', t_{2^{n_c}} = (c,1))$ be the ordered vertex set of $[(0,0),(c,1)]'$. Then $(\pi_c^{-1})(\{t_j\}) = \{v_j\} = \{w_j\}$, a single vertex, and $(\pi_c^{-1})(\{t_j'\}) = \{v_j,w_j'\}$, a set of two vertices. Finally, if $j \neq 0, 2^{n_c}$, the vertex $v_j=w_j$ has degree four, so every point in $\Gamma(c)$ is at most two edge lengths away from a vertex of degree four. Thus, $\pi_c: \Gamma(c) \to [(0,0),(c,1)]'$ satisfies the conditions listed for $\pi_i^{i+1}$ in Axioms \ref{ax:gendiamsys2}, \ref{ax:gendiamsys3}, and \ref{ax:gendiamsys7}.

For any $\alpha_0 \in \R \setminus \{0\}$, $b_0 \in B$, and $A \sbs B$, we let $\alpha_0 A + b_0$ be the image of $A$ under the invertible similarity $b \mapsto \alpha_0 b + b_0$. $\alpha_0 \Gamma(c) + b_0$ and $(\alpha_0[(0,0),(c,1)] + b_0)'$ inherit graph structures from $\Gamma(c)$ and $[(0,0),(c,1)]'$, respectively, and there is also an induced map $\alpha_0 \pi_c + b_0: \alpha_0 \Gamma(c) + b_0 \to (\alpha_0[(0,0),(c,1)] + b_0)'$ that, like $\pi_c$, satisfies Axioms \ref{ax:gendiamsys2}, \ref{ax:gendiamsys3}, and \ref{ax:gendiamsys7}.

\subsubsection{Inductive Construction of $X_i$}
\label{sss:indconstr}
For the base case, let $X_0 = \{0\} \cross I \sbs B \oplus \R$. For the inductive hypothesis, assume that the inverse system $X_i \overset{\pi^i_{i-1}}{\to} X_{i-1} \dots \overset{\pi^1_0}{\to} X_0$ and $X_{i-1}'$ have been constructed and satisfy Axioms \ref{ax:gendiamsys1}-\ref{ax:gendiamsys7} from Definition \ref{def:gendiamsys}. For $e \in E(X_i)$, let $v_0(e)$ and $v_1(e)$ denote the terminal vertices of $e$. Assume that the inverse system satisfies the additional properties: 

\begin{enumerate}[label=(P\arabic*)]
\item \label{ax:addprop1} For all $e \in E(X_i)$, $e$ equals the line segment joining $v_0(e)$ to $v_1(e)$. That is, $e = [v_0(e),v_1(e)] := \{(1-t)v_0(e)+tv_1(e): t \in [0,1]\}$.
\item \label{ax:addprop2} For all $e \in E(X_i)$, $e$ is parallel to an associated vector $(c,1) \in C \cross \{1\}$. That is, $v_1(e) - v_0(e) = \alpha (c,1)$ for some $\alpha \in \R$ and $c \in C$. Furthermore, $\alpha = 2^{-n_i}$ for some $n_i \geq 1$. $n_i$ depends on $i$ but not on $e$. It follows that every edge of $X_i$ has length $2^{-n_i}$.
\end{enumerate}

Now we need to construct $X_{i+1}$, $X_i'$, and $\pi^{i+1}_i: X_{i+1} \to X_i'$. Let $e \in E(X_i)$, and $c \in C$ and $n_i \geq 1$ such that $v_1(e) - v_0(e) = 2^{-n_i} (c,1)$. Subdivide $e$ into into 3 subedges, the middle one having length $\frac{1}{2}|e|$, and the terminal ones having length $\frac{1}{4}|e|$. Let $e_0$ and $e_1$ denote the terminal subedges, and $e_{1/2}$ the middle subedge. Note that, for any $x \in e_{1/2}$ and $y \in X_i \setminus e$,
\begin{equation} \label{eq:mindist}
d_i(x,y) \geq \frac{|e|}{4} = 2^{-(n_i+2)}
\end{equation}

Let $\delta' = \frac{\delta+\delta_i}{2}$, so that $\delta < \delta' < \delta_i$. Choose $N$ to be large enough so that
\begin{equation} \label{eq:largesubdiv}
2^{-N} \leq (\delta_i-\delta')2^{-2}
\end{equation}
Subdivide $e_{1/2}$ into $2^N$ edges of equal length. So now $e$ is divided into a total of $2^N + 2$ subedges, and two of them, $e_0$ and $e_1$, are marked as terminal subedges. Doing this for every $e \in E(X_i)$ gives us a subdivison $\tilde{X}_i$ of $X_i$. Let $f$ be a subedge of $e_{1/2}$. Then $v_1(f)-v_0(f) = 2^{-(n_i+1+N)}(c,1)$. We create $X_{i+1}$ by replacing $f$ with the graph $2^{-(n_i+1+N)}\Gamma(c) + v_0(f)$, which has the same vertices as $f$. Thus, $X_{i+1}$ consists of the union of $e_0,e_1,2^{-(n_i+1+N)}\Gamma(c) + v_0(f)$ over all $f \sbs e_{1/2}$ and $e \in E(X_i)$, with each $e_0$ and $e_1$ subdivided into subedges so that every edge of $X_{i+1}$ has equal length. $X_{i+1}$ satisfies \ref{ax:addprop1} and \ref{ax:addprop2}. Since there are only finitely many $e \in E(X_i)$, and thus finitely many $c \in C$ associated to $e$, we may choose the subdivision parameter $n_c$ of Section \ref{sss:modelgraph} independent of $c$.

$X_i'$ is simply the subdivision of $\tilde{X}_i$ into subedges all having length the same as any edge of $X_{i+1}$. For any $e_0$, $e_1$, and $f \sbs e_{1/2}$, let $e_0'$, $e_1'$, and $f'$ denote the subdivisions in $X_i'$. Let $2^{-(n_i+1+N)}\pi_c + v_0(f): 2^{-(n_i+1+N)}\Gamma(c) + v_0(f) \to f'$ be the map defined in Section \ref{sss:modelgraph}. We paste all these maps along with all the identity maps $e_0 \to e_0'$, $e_1 \to e_1'$ together to obtain the quotient map $\pi_i^{i+1}: X_{i+1} \to X_i'$. Then $\pi_i^{i+1}$ satisfies Axioms \ref{ax:gendiamsys2}, \ref{ax:gendiamsys3}, and \ref{ax:gendiamsys7} because each map $2^{-(n_i+1+N)}\pi_c + v_0(f)$ does.

The map $\pi_i^{i+1}$ is a 1-Lipschitz quotient with respect to the metrics $d_{i+1}$ and $d_i$. Furthermore, it has the property that, if $x,y \in X_{i+1}$ and $\pi^{i+1}_i(x)$ and $\pi^{i+1}_i(y)$ do not belong to the same edge of $\tilde{X}_i$, then
\begin{equation} \label{eq:samedist}
d_{i+1}(x,y) = d_i(\pi^{i+1}_i(x),\pi^{i+1}_i(y))
\end{equation}

Set $\delta_{i+1} := \min_c(\delta_c,\delta') > \delta$, where the minimum is over each $(c,1)$ associated to an edge $e$ of $X_i$. Since there are only finitely many edges of $X_i$, the minimum is well-defined and $\delta_{i+1} > \delta$. We now check that $X_{i+1}$ is $\delta_{i+1}^{-1}$-quasiconvex.

Let $x,y \in X_{i+1}$. First consider the case $\pi_i^{i+1}(x)$ and $\pi_i^{i+1}(y)$ belong to the same edge $f$ of $\tilde{X}_i$, with $v_1(f) - v_0(f) = 2^{-n_i}(c,1)$ for some $c \in C$. Then $x$ and $y$ both belong to $2^{-(n_i+1+N)}\Gamma(c) + v_0(f)$, on which the intrinsic distance is $\delta_c^{-1}$-quasiconvex, so the desired conclusion holds in this case.

Now assume $\pi_i^{i+1}(x)$ and $\pi_i^{i+1}(y)$ do not belong to the same edge of $\tilde{X}_i$ but do belong to the same edge of $X_i$. Then the intrinsic and extrinsic distance between $x$ and $y$, and the intrinsic and extrinsic distance between $\pi_i^{i+1}(x)$ and $\pi_i^{i+1}(y)$ are all equal.

Finally, assume $\pi_i^{i+1}(x)$ and $\pi_i^{i+1}(y)$ do not belong to the same edge of $X_i$. We consider two subcases: both $x$ and $y$ belong to terminal subedges of $e,f \in E(X_i)$, or one does not belong to a terminal subedge. In the first case, if both $x$ and $y$ belong to terminal subedges of $X_i$, then $\pi^{i+1}_i$ acts identically, on $x$ and $y$, and so
$$d_{i+1}(x,y) = d_i(\pi_i^{i+1}(x),\pi_i^{i+1}(y)) \leq \delta_i^{-1}\|\pi_i^{i+1}(x)-\pi_i^{i+1}(y)\| = \|x-y\|$$
by the inductive hypothesis and so the conclusion holds. Now assume, without loss of generality, that $\pi_{i}^{i+1}(x) \in e_{1/2}$ for some $e \in E(X_i)$ and $y \in X_{i+1} \setminus (\pi^{i+1}_i)^{-1}(e)$. Then we get
\begin{equation} \label{eq:mindistapp}
d_i(\pi_i^{i+1}(x),\pi_i^{i+1}(y)) \geq \|\pi_i^{i+1}(x)-\pi_i^{i+1}(y)\| \geq \delta_i d_i(x,y) \overset{\eqref{eq:mindist}}{\geq} \delta_i\frac{|e|}{4} = \delta_i 2^{-(n_i+2)}
\end{equation}
Since $\pi_i^{i+1}$ acts identically on the vertices of $\tilde{X}_i$, the $d_{i+1}$ diameter of any fiber of $\pi_i^{i+1}$ is at most the length of an edge of $\tilde{X}_i$, which is $2^{-(n_i+1+N)}$. This implies
\begin{equation} \label{eq:fiberdiam2}
\|\pi_i^{i+1}(x) - x\|, \|\pi_i^{i+1}(y) - y\| \leq 2^{-(n_i+1+N)}
\end{equation}
Thus,
$$\|x-y\| \geq \|\pi_i^{i+1}(x) - \pi_i^{i+1}(y)\| - \|\pi_i^{i+1}(x) - x\| - \|y - \pi_i^{i+1}(y)\|$$
$$\overset{\eqref{eq:fiberdiam2}}{\geq} \delta_id_i(\pi_i^{i+1}(x),\pi_i^{i+1}(y)) - 2^{-(n_i+N)}$$
$$\overset{\eqref{eq:largesubdiv}}{\geq} \delta_id_i(\pi_i^{i+1}(x),\pi_i^{i+1}(y)) - (\delta_i-\delta')2^{-(n_i+2)}$$
$$\overset{\eqref{eq:mindistapp}}{\geq} \delta_id_i(\pi_i^{i+1}(x),\pi_i^{i+1}(y)) - (\delta_i-\delta')d_i(\pi_i^{i+1}(x),\pi_i^{i+1}(y))$$
$$= \delta'd_i(\pi_i^{i+1}(x),\pi_i^{i+1}(y)) \overset{\eqref{eq:samedist}}{=} \delta'd_{i+1}(x,y) \geq \delta_{i+1}d_{i+1}(x,y)$$
\end{proof}

\section{Questions}
\label{sec:questions}
As far as we are aware the following questions remain open. A positive answer to (Q1) implies a positive answer to (Q2), a positive answer to (Q2) implies a positive answer to (Q3) (assuming the metric space is separable), and a positive answer to (Q3) implies positive answers to (Q4) and (Q5). In the following, we always mean ``complete metric space(s)" when we say "metric space(s)" (there are easy counterexamples if completeness is not assumed).

\begin{enumerate}[label=(Q\arabic*)]
\item \label{Q:1} Is the differentiation nonembeddability criterion a necessary condition for the non-biLipschitz embeddability of metric spaces into RNP spaces?
\item \label{Q:2} Is some weak form of the differentiation nonembeddability criterion, such as that of Theorem \ref{thm:mainthmssummary}, a necessary condition for the non-biLipschitz embeddability of metric spaces into RNP spaces?
\item \label{Q:3} Are the only obstructions to biLipschitz embeddability of metric spaces into RNP spaces local?
\item \label{Q:4} Does every discrete metric space biLipschitz embed into some RNP space?
\item \label{Q:5} Is there a universal constant $C$, such that if a metric space biLipschitz embeds into an RNP space, then it $C$-biLipschitz embeds into an (possibly larger) RNP space?
\end{enumerate}

Technically, the differentiation nonembeddability criterion is for metric measure spaces, not metric spaces, so we need to be more specific about \ref{Q:1} (and similarly for \ref{Q:2}): if a complete metric space $M$ does not embed into any RNP space, does there exist a Borel measure $\mu$ on $M$ so that $(M,\mu)$ satisfies the differentiation nonembeddability criterion defined in Section \ref{ss:histback}?

\ref{Q:3} can be stated more specifically as: if every point in a metric space has a neighborhood that biLipschitz embeds into an RNP space, does the entire metric space embed into an RNP space?

That \ref{Q:1} implies \ref{Q:2} and \ref{Q:3} implies \ref{Q:4} are immediate.

To see that \ref{Q:1} implies \ref{Q:3}, suppose we have a separable metric space that locally embeds into RNP spaces. By using a partition of unity type argument, and by taking a countable $\ell_1$-direct sum of the RNP spaces, we can find a globally defined Lipschitz map into a single RNP space that is locally biLipschitz. Then the blowup of this map at any point where it exists must also be biLipschitz, and so no differentiation criterion can hold. Thus, by \ref{Q:1}, the entire metric space must biLipschitz embed into some RNP space.

The statement of \ref{Q:2} is not specific enough to prove that \ref{Q:2} implies \ref{Q:3}, but for any reasonable notion of differentiability, the same argument of \ref{Q:1} implies \ref{Q:3} should work.

Now we show that \ref{Q:3} implies \ref{Q:5}. We first need the following result: there is a constant $C$ such that, for any metric space $X$, if every bounded subset of $X$ biLipschitz embeds into an RNP space with distortion at most $D$, then the entire space embeds into an RNP space with distortion at most $CD$. To see this, pick a base point $0 \in X$, RNP spaces $B_n$, and $D$-biLipschitz embeddings $\phi_n: B_{2^n}(0) \to B_n$. Define a new map $\phi: X \to (\oplus B_n)_1 \oplus \R$ (note that the target has RNP) by $\phi(x) = \left( \frac{d(x,0)-2^{i-1}}{2^{i-1}}\phi_{i-1}(x) + \frac{2^{i}-d(0,x)}{2^{i-1}}\phi_i(x),d(0,x) \right)$ if $2^{i-1} \leq d(0,x) \leq 2^i$. It is shown in the proof of Theorem 1.2 from \cite{Os12} that there is a constant $C$ such that $\phi$ is $CD$-biLipschitz. This ``gluing" procedure was also used earlier in the proof of Theorem 1.1 of \cite{Bau07}. Now assume \ref{Q:5} is false. For each $N \in \mathbb{N}$, pick a metric space $X_N$ that biLipschitz embeds into some RNP space but with distortion never less than $N$. By the previous discussion, there is a universal constant $c$ such that if every bounded subset of $X_N$ biLipschitz embeds into an RNP space with distortion less than $cN$, then the entire space would biLipschitz embed into an RNP space with distortion at most $N$, a contradiction. Thus, there is a bounded subset $Y_N$ of $X_N$ that biLipschitz embeds into an RNP space, but never with constant less than $cN$. Taking a metric disjoint union of these bounded spaces $\{Y_N\}_{N=1}^\infty$ yields a space failing \ref{Q:3}.



\bibliographystyle{amsalpha}
\bibliography{thickfamilypaperbib}

\end{document}